\documentclass[11pt]{article}
\usepackage{amsfonts,amssymb,amsmath}
 \usepackage{amsthm}
 
 \usepackage{graphicx}
\usepackage{subcaption}
\usepackage{pgfplots}
  \usepackage{pgfplots}
\usepackage{xcolor}

\usepackage{amssymb,amsmath,amsfonts,amsthm,amscd,amstext,mathtools}
\usepackage[T1]{fontenc}
\usepackage[english]{babel}
\usepackage[numbers,sort&compress]{natbib}  
\usepackage{hyperref,doi}
\usepackage{graphicx}
\usepackage{perpage} 
\usepackage{url}
\usepackage{color}
\usepackage[a4paper,scale={0.72,0.74},marginratio={1:1},footskip=7mm,headsep=10mm]{geometry}
\usepackage{enumerate}
\usepackage{graphicx} 
\usepackage{nlatexmac}

\newcommand{\ee}{\epsilon}
\newcommand{\kzero}{{\zeta}}
\newcommand{\fddn}{\xRightarrow[ N]{f.d.d}}
 
\DeclareMathOperator*{\argmin}{argmin}

 \def\botcaption#1#2{\medskip\centerline{{\scshape #1.}\kern8pt
 {\rm #2}}\bigskip}

 \makeatletter
 \@addtoreset{equation}{subsection}
 \makeatother

 \makeatletter
 \@addtoreset{enunciato}{section}
 \makeatother

 \newcounter{enunciato}[subsection]

 \newtheorem{ittheorem}{Theorem}
 \newtheorem{itlemma}{Lemma}
 \newtheorem{itproposition}{Proposition}
 \newtheorem{itdefinition}{Definition}
 \newtheorem{itremark}{Remark}
 \newtheorem{itclaim}{Claim}

 \newenvironment{theorem}{\addtocounter{enunciato}{1}
 \begin{ittheorem}}{\end{ittheorem}}

 \newenvironment{lemma}{\addtocounter{enunciato}{1}
 \begin{itlemma}}{\end{itlemma}}

 \newenvironment{proposition}{\addtocounter{enunciato}{1}
 \begin{itproposition}}{\end{itproposition}}

 \newenvironment{definition}{\addtocounter{enunciato}{1}
 \begin{itdefinition}}{\end{itdefinition}}

 \newenvironment{remark}{\addtocounter{enunciato}{1}
 \begin{itremark}}{\end{itremark}}

 \newenvironment{claim}{\addtocounter{enunciato}{1}
 \begin{itclaim}}{\end{itclaim}}

 \newcommand{\bl}[1]{\begin{lemma}\label{#1}}
 \newcommand{\el}{\end{lemma}}

 \newcommand{\br}[1]{\begin{remark}\label{#1}}
 \newcommand{\er}{\end{remark}}

 \newcommand{\bt}[1]{\begin{theorem}\label{#1}}
 \newcommand{\et}{\end{theorem}}

 \newcommand{\bd}[1]{\begin{definition}\label{#1}}
 \newcommand{\ed}{\end{definition}}

 \newcommand{\bcl}[1]{\begin{claim}\label{#1}}
 \newcommand{\ecl}{\end{claim}}

 \newcommand{\bp}[1]{\begin{proposition}\label{#1}}
 \newcommand{\ep}{\end{proposition}}

 \newcommand{\bc}[1]{\begin{corollary}\label{#1}}
 \newcommand{\ec}{\end{corollary}}

 \newcommand{\bpr}{\begin{proof}}
 \newcommand{\eprz}{\end{proof}}

 \newcommand{\bi}{\begin{itemize}}
 \newcommand{\ei}{\end{itemize}}

 \newcommand{\ben}{\begin{enumerate}}
 \newcommand{\een}{\end{enumerate}}
 
 \newcommand{\x}{{\bf x}}
 \newcommand{\y}{{\bf y}}
 \newcommand{\s}{{\bf s}}

 \def\botcaption#1#2{\medskip\centerline{{\scshape #1.}\kern8pt
 {\rm #2}}\bigskip}

 \def \ba {\begin{array}}
 \def \ea {\end{array}}

 \def \Z {{\mathbb Z}}
 \def \R {{\mathbb R}}
 
 \def \N {{\mathbb N}}
 
 \def \E {{\mathbb E}}

 \def \cH {{\mathcal H}}

 \def \cA {{\mathcal A}}

 \def \cB {{\mathcal B}}
 \def \cF {{\mathcal F}}
 
 \def \cC {{\mathcal C}}
 
 \def \cS {{\mathcal S}}
 \usepackage{authblk}

 \def \ind {{1}}
 \def \gep {{\varepsilon}}

\begin{document}

\title{ Non-local random deposition models for earthquakes and propagation in some random media}

\author{
    Philippe Carmona$^{1}$, François Pétrélis$^{2}$ and Nicolas Pétrélis$^{1}$ \\
    {\small $^{1}$Laboratoire de Mathématiques Jean Leray UMR 6629, Université de Nantes, 2 Rue de la Houssinière, BP 92208, F-44322 Nantes Cedex 03, France} \\
    {\small \texttt{philippe.carmona@univ-nantes.fr, nicolas.petrelis@univ-nantes.fr}} \\
    {\small $^{2}$Laboratoire de Physique de l'Ecole Normale Supérieure, 24 rue Lhomond, 75005 Paris, France} \\
    {\small \texttt{petrelis@phys.ens.fr}}}

%
%



%
%



\maketitle

 \abstract{
 In the present paper, we investigate a new class of  non-local random deposition models, initially introduced by physicists to investigate the field of  mechanical constraints (stress) applied along a line or on a given area 
 located in a seismic zone. The non-local features are twofold. First, the falling objects have random and heavy-tailed dimensions.
Second, the locations where the objects are falling are at least for some of the models that we consider, depending on the shape of the surface before the deposition. 
Let us be more specific, we consider $(h_N)_{N\in \N}$ a sequence of random  $(d+1)$-dimensional surfaces defined on $[0,D]^d$ for $d\in \{1,2\}$. 
The process $h_{N}$ is obtained by adding to $h_{N-1}$ an object  
$$\s\in [0,D]^d \mapsto Z_{N}^{\alpha-d} \psi\Big(\frac{v_{Y_N}(\s)}{Z_N}\Big),$$
 where $Z=(Z_i)_{i\in \N}$ is an  i.i.d. sequence of Pareto random variables with parameter $\beta>0$  introduced to tune the horizontal size of the falling object, where $\alpha>0$ provides the links between the width and the height of the object, where  $\psi:[0,\infty)\mapsto \mathbb{R}^+$  determines the global shape of the object, where $v_{\y}(\x)$ is the distance between $\x$ and $\y$ on the torus and where $Y=(Y_i)_{i\in \N}$ are random variables in $[0,D]^d$ that settle the deposition location of each falling object. 
 
 In the present paper we focus our attention on three variations of this model.
 First, the rand-model which corresponds to $Y$ being a sequence of i.i.d. random variables distributed uniformly on $[0,D]^d$. 
Then,  the min-model, that introduces an important property of the physics of earthquakes but is also more difficult to tract to the extent that a strong correlation appears 
 between the $N$-th falling object and the shape of the profile $h_{N-1}$. Finally we also consider another variant of the rand-model, called the stellar model, which describes the energy absorption of a propagating field by random objects. It models for instance the intensity of the microwaves  emitted by stellar clouds 
and measured   at the Earth surface. 

For those three models, our results identify the limit in law of $(h_N)_{N\in \N}$ viewed as a sequence of continuous random functions rescaled properly.  We also determine the limit in law of the fluctuations of $(h_N)_{N\in \N}$.

   }

{\bf }
\medskip

 \tableofcontents



\section{Outline of the paper}
In Section \ref{Intro} we present rigorously the heavy-tailed random deposition model (referred to as rand model) that we are going to study all along the paper. 
Subsequently, with Sections \ref{absorb} (resp. \ref{stellar} and \ref{earthquakemod}) we explain how this model and two of its variants, i.e., the stellar model and the min model, are  good 
matches to study the absorption of energy of a wave propagating in a non homogeneous medium (resp.,  the emission of radiations by interstellar clouds, resp. the 
stress field in a seismic zone). 
Note that with Section \ref{rela}, we briefly recall why the non-local features of our models make them different from the deposition models that have been investigated in the past.
Section \ref{Resultat} is dedicated to the main results of the paper. With Theorems \ref{speedfront} and \ref{speedfrontvariant} we provide the limiting distribution of the sequence of random functions 
considered for each of the three models. With Theorems \ref{fluctconver} and \ref{fluctconverstel}, in turn, we provide the asymptotics of their fluctuations (except for the min model).
Note that for both the random functions and their fluctuations, depending on the parameters of the model, we observe different regimes separated by critical values 
that are calculated explicitly.  At the end of section \ref{Resultat}, in Section \ref{perspec} we give some open problems and conjectures mostly related to the fluctuations of the min model. With Section \ref{tool-for-distr-conv}, we provide some tools to prove the convergence in distributions of random functions. Those tools will be used subsequently to prove Theorems \ref{speedfront} and \ref{speedfrontvariant} in Section \ref{Preuves1} and to prove  Theorems \ref{fluctconver} and \ref{fluctconverstel} in Sections \ref{Preuves} and \ref{Preuvesstel}.  


\section{Introduction: }\label{Intro}

Objects that vary greatly in size, and whose distribution can be modelled by a power law, frequently occur in nature.

This is the case of aerosols in the atmosphere \cite{JUNGE1955,BOUCHER2015}, or interstellar clouds in space \cite{ELMEGREEN1996}. In both cases, these objects are formed by collision and fragmentation processes, which can be described by Schmoluchowski-type equations \cite{WHITE1982,KRAPIVSKI2010}. Note that the power-law behavior of  the object size distributions is  determined by the shape of the reaction rates in these equations.

Earthquakes also have certain power-law properties. At the boundary between two tectonic plates, slow mantle movements accumulate stresses that are suddenly released during an earthquake by the fast movement of a zone at the plate boundary \cite{SCHOLZ2019}. The surface of the moving zone fluctuates, and its distribution is known to be a power-law \cite{KANAMORI2004,KAWAMURA2012}. Recent works have proposed an explanation for the origin of these distributions, based on the self-similar nature of the force field at the plate boundary. This provides an explanation for statistical properties of earthquakes such as the Gutenberg-Richter law for the released energy or the Omori-Utsu law for the number of aftershocks \cite{PETRELIS2023,PETRELIS2024}. 

Several situations of interest for the objects introduced above can be described as sums over sets of objects. As we describe in sections 1.2 and 1.3 below,  this is the case, for example, with the emission of microwaves by interstellar clouds, or for the absorption by aerosols of radiation which rays propagate in a straight line. In the case of earthquakes, the stress field results from the accumulation of stresses created by all past events, and can therefore be modeled as the summation of stress variations caused by previous events, see section 1.4.

This summation over sets of objects is akin to considering a deposition process, the particularity of which is that object sizes are power-law distributed. This process  forms an interface corresponding to different physical observables, depending on the problem under consideration.
In the case of microwaves emitted by the interstellar medium  or in the case of absorption by aerosols,  these are the spatial variations in radiation intensity when measured on Earth in different directions towards space.
For earthquakes, it is the stress field measured as a function of position in the plane separating the two sides of a fault.

We will now describe in more detail how to model these different deposition processes. Variants are considered according to the distribution of object positions and to the scalar or vector nature of the objects.

We call  rand-model the one where object positions are sampled uniformly at random and objects are scalar. The stellar model corresponds to vector objects. Finally, the min-model considers that objects are deposited at the minimum of the interface, so at a  position that depends on the  whole shape of the interface. This last, highly non-local problem is the most complicated, but introduces an effect that is important for earthquake modeling.


\subsection{Heavy-tailed random deposition  model}\label{defrandmod}

We consider a sequence of real functions (called profiles)  $h:=(h_n)_{n\in \N\cup\{0\}}$ defined on $[0,D]^d$
which are obtained recursively through a sequence of random transformations.  
In order to define rigorously these transformations, several
mathematical objects are  required:

\begin{enumerate}
\item \emph{Shape of the profile transformation}\label{def:hng}
Given an integer $n\in\N$, $n\ge 1$  we let $\cH^n$ be  the space of functions
$\psi:[0,+\infty[\to[0,+\infty[$ such that
\begin{itemize}
  \item $\psi(0)=1$ and $\psi(x)=0$ for $x\ge 1$.
  \item $\psi$ has an order $n-1$ derivative on
  $[0,1]$ which is Lipschitz, that is there exists 
   $C\in(0,\infty)$ such that
  \begin{equation}\label{eq:psihol}
    \valabs{\psi^{(n-1)}(x) -\psi^{(n-1)}(y)} \le C \valabs{x-y}\qquad
    (x,y \in [0,1]).
  \end{equation}
  \item the function $\psi$ is n times differentiable at $0$,
  $\psi^{(n)}(0) \neq 0$ and  $\psi^{(k)}(0)=0$ for $1\le k\le n-1$.
\end{itemize}
For example the function $\psi(x)=(1-x)^3 \un{0\le x\le 1}$ is in
$\cH^3$.

We let $\cH^0$ be the space of functions $\psi:[0,\infty)\mapsto [0,\infty)$ continuous such that $\psi(0)=1$ and 
$\psi(x)=0$ for $x\in [1,\infty)$.

\item The \emph{width} on which a profile transformation occurs is \emph{heavy tailed}. Therefore, we let  $Z:=(Z_i)_{i\geq 1}$ be an i.i.d sequence of random variables defined on a probability space $(\Omega,\cA,P)$ and such that 
$Z_1$ is Pareto distributed with parameter $\beta-1$ ($\beta>1$), i.e., with density $\ind_{[1,\infty)}(u) (\beta-1)/u^\beta$, $u\in \R$.
\item The \emph{center} of a profile transformation is picked uniformly on $[0,D]^d$. Thus, we let  $Y:=(Y_i)_{i\geq 1}$ be 
an i.i.d. sequence of random variables defined as well on  $(\Omega,\cA,P)$, independent of $Z$ such that $Y_1$ is \emph{uniformly distributed on $[0,D]^d$}.  
\item We will use periodic boundary conditions and consequently, for $\y \in \R^d$, we define $v_{\y}:\R^d\mapsto[0,dD]$ 
which indicates the distance between a given $\s\in \R^d$ and the lattice $\y+\Z^d D$, i.e., 
\begin{align}\label{defaltv}
v_{\y}(\s)&:=\min\{|\s-\y+j D|_1, \ j\in \Z^d\}\\
\nonumber &=\text{dist}(\s,\y+\Z^d D), \quad \s\in \R^d,
\end{align}
where we have use the 1-norm on $\R^d$, i.e., $|x-y|_1:=|x_1-y_1|+\dots+|x_d-y_d|$.
We note that $v_{\y}(\s)=v_{\s}(\y)$ and we can draw the graph of $v_{\y}(\cdot)$ in dimension $d=1$ on Figure \ref{fig-rho*def}.  

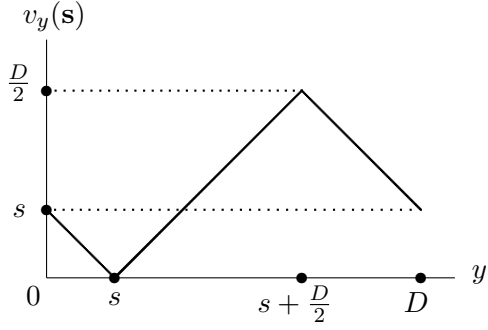
\begin{figure}\label{graphvy}
\begin{center}
\setlength{\unitlength}{0.45cm}
\begin{picture}(8,8)(0,1)
\put(0,0){\line(9,0){12}}
\put(0,0){\line(0,7){7}}
{\thicklines
\qbezier[30](0,5.5)(4,5.5)(7.5,5.5)
\qbezier[40](0,2)(5,2)(10.7,2)
}
{\thicklines
\qbezier(0,2)(1,1)(2,0)
\qbezier(2,0)(3.5,1.5)(7.5,5.5)
\qbezier(7.5,5.5)(9.5,3.5)(11,2)

}
\put(-.6,-.8){$0$}
\put(1.8,-.8){$s$}
\put(0,2){\circle*{.35}}
\put(2,0){\circle*{.35}}
\put(7.5,0){\circle*{.35}}
\put(11,0){\circle*{.35}}

\put(0,5.5){\circle*{.35}}
\put(12.5,0){$y$}
\put(-.7,7.5){$v_y(\s)$}
\put(-1.2,5.5){$\frac{D}{2}$}
\put(-1,1.8){$s$}
\put(10.5,-1){$D$}
\put(6.2,-1){$s+\tfrac{D}{2}$}
\end{picture} 
\vspace{1cm}
\end{center}
\caption{In dimension $d=1$: graph of  $y\mapsto v_y(s)$ avec $s\in [0,\frac{D}{2}]$}
\label{fig-rho*def}  
\end{figure}

\end{enumerate}

\medskip

With these tools, for $i\in \N$,  we can define rigorously the $i$-th profile transformation as
\begin{equation}\label{elemprotrans}
X_i(\x):=Z_{i}^{\alpha-d} \psi \Big[\frac{v_{Y_{i}}(\x)}{Z_{i}}\Big], \quad \x\in [0,D]^d.
\end{equation}
where $\alpha\in (0,\infty)$ is a parameter which value is set by the physics of the  considered model.

By summing out the $N$ first profile transformations in \eqref{elemprotrans}, we obtain the sequence of profiles $h:=(h_N)_{N=0}^{\infty}$ as
\begin{align}\label{defrand}
h_0(\x)&=0, \quad \x\in [0,D]^d,\\
\nonumber h_{N+1}({\bf x})&=h_{N}(\x)+X_{N+1}(\x),\quad   N\in \N_0,\  \x\in [0,D]^d.
\end{align}
All along this paper, we will study the asymptotics of $h=(h_N)_{N=0}^\infty$ as the number of profile transformation diverges ($N\to \infty$).
Thus, depending on the values taken by the parameters $\alpha,\beta, n$ and on the dimension $d$, we will investigate   the limit in distribution of $(h_N)_{N\in \N}$ seen as sequences of real functions on $[0,D]^d$ (properly rescaled).
\smallskip  

We will  also  try to determine the asymptotics of  $f:=(f_N)_{N=0}^\infty$ which is the sequence of profile fluctuations. Thus,  for $N\to \infty$, we consider  $f_N:[0,D]^d\mapsto \R$ defined as   $f_N(\x):=h_N(\x)-h_N(0)$, i.e., 
\begin{align}\label{deffn}
\nonumber f_{N+1}(\x)&=f_{N}(\x)+X_N(\x)-X_N(0)\\
&=f_{N}(\x)+ Z_{N+1}^{\alpha-d} \  \bigg[ \psi\Big(\frac{v_{Y_{N+1}}(\x)}{Z_{N+1}}\Big)-\  \psi\Big(\frac{v_{Y_{N+1}}(0)}{Z_{N+1}}\Big)\bigg], \quad \x\in  [0,D]^d.
\end{align}



%
%

\subsection{Physical motivations for the rand-model: energy propagation through absorbing objects.} 
\label{absorb}

	When a wave propagates in a medium that contains absorbing objects, its energy decreases after every encounter with such an absorbing object.

First, we assume that the wave propagates in a straight line perpendicular to a window $[0,D]^d$ ($d=1$ or $d=2$). Second, we assume that the decrease of energy at each encounter is proportional to the width of the absorbing object that is crossed by the straight line and third that the positions of the objects are sampled uniformly at random, then the opposite of the variation of energy  
is described by the rand-model, as illustrated in Figure \ref{figsup1}.

\begin{remark}\rm The standard law for the absorption of energy is that of an exponential decay with rate proportional to a coefficient $a$ multiplied by the length of propagation $l$. If the effect is small, we can develop the exponential as  $e^{(-a l)}\simeq 1-a l$, and we recover an effect proportional to the length $l$. If the effect is not small, then we consider the log of the energy instead which also changes by $-a l$ after propagation over a distance $l$.
\end{remark}

This process provides for instance a simple model for the spatial fluctuations of the electromagnetic energy as measured at Earth surface when the electromagnetic radiations have propagated through the atmosphere and interacted with aerosols.

\begin{center}
\begin{figure}[htb!]
\includegraphics[width=10cm]{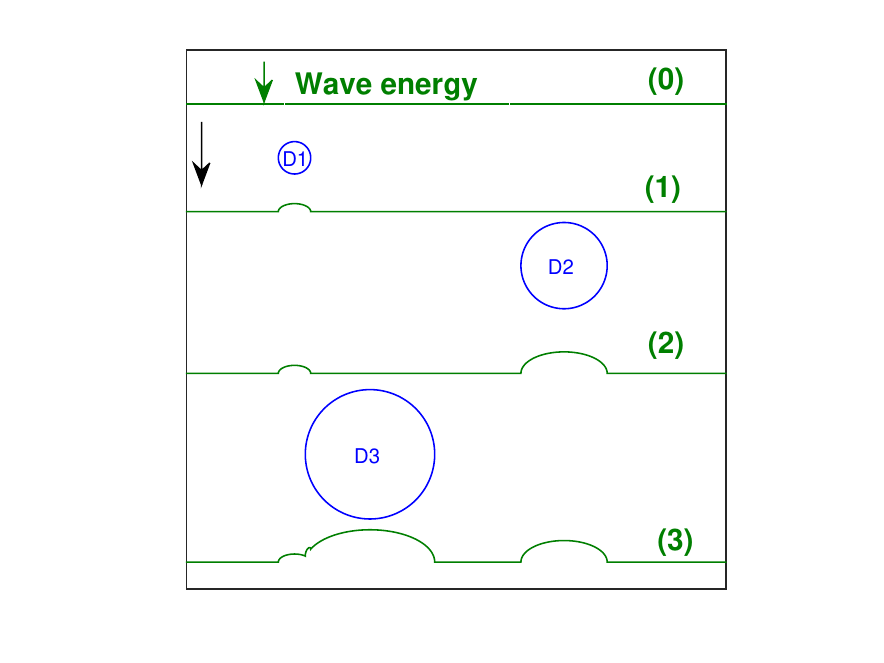}
\caption{Schematics of {\bf minus} the variation of the energy of a wave that propagates in  straight line from the top to the bottom parallel to the black arrow. The wave energy is displayed as a green line and consists 
in any electromagnetic radiation assumed to propagate in a straight line and absorbed by aerosols in the atmosphere.  It is decreased  when the ray  interact with an aerosol. The absorption by aerosol is assumed to be proportional to their width and independent of the energy, as expected when their effect is small.  The energy after the $(i)$-th encounter is displayed for $i=0$ to $3$. Note that the process is quite generic and this figure describes also the emission of any quantity that propagates in straight line and encounters objects that  are additive sources of amplitude proportional to their width.} 
\label{figsup1}
\end{figure}
\end{center}

\subsection{The stellar-model} \label{stellar}

A dual situation concerns a medium in which objects radiate energy. This is the case for instance of interstellar clouds in space that radiate microwaves.
The intensity of a cloud radiation is proportional to the 
width of the cloud. To adapt our model to the present phenomenon we recall \eqref{defrand} and we choose $d=2$.
The sequence of random variables  $Z=(Z_i)_{i=1}^\infty$ provides the size of interstellar clouds indexed in $\N$. The observations show that the size of the clouds varies greatly and their distribution  is considered to be a power-law as discussed in \cite{ELMEGREEN1996}.  The second sequence of random variables  $Y=(Y_i)_{i=1}^\infty$ provides the location of the center of each cloud. Finally, the shape of a cloud being an ellipsoid and the signal measured at the Earth surface being proportional to the width of the ellipsoid projected on the plane normal to the line of sight, we set 
$\psi(u)=(1-u^2)\ind_{[0,1]}(u)$ for $u\in \R$.

The statistical characterization of this interstellar emission is of prime importance to experiments searching the signature of primordial gravitational waves in the cosmic microwave background (CMB) polarization.
These waves are polarized due to the presence of a large scale magnetic field of scale comparable or larger than the one of the cloud. Each cloud thus emits a microwave radiation aligned with a direction that varies from one cloud to another. We assume that these directions are uncorrelated. At the earth surface, measurements are sensitive either to the total energy or to the total field. The former is described by the rand-model whereas the later can be defined as follows. This is the stellar-model. 

We let  $\cS_{2}$ be the unitary sphere of dimension $3$, i.e., 
\begin{equation}\label{spher}
\cS_{2}:=\{(x_1,x_2,x_3)\colon\, x_1^2+x_2^2+x_3^2=1\}.
\end{equation}
We let $(\widetilde \Theta_i)_{i=1}^\infty$ be an i.i.d. sequence of random variables taking their values on $S_2$ and we denote by $\theta_0:=E(\widetilde \Theta_1)$. Thus, it suffices to set $\Theta_i:=\widetilde \Theta_i-\theta_0$ for $i\in \N$
to obtain an i.i.d. sequence of bounded centered random variables $(\Theta_i)_{i\in \N}$.
Then, the radiation that is measured at $\x\in \R^2$ is  $\theta_0 h_N( \x)+h^{\text{st}}_N(\x)$ where 
\begin{align}\label{defrandmoddim3}
h^{\text{st}}_N({\bf x})&:=\sum_{j=1}^N X_j(\x)\, \Theta_j, \quad \x\in [0,D]^d.
\end{align}

\subsection{Physical motivations in the context of earthquakes. The min-model.} \label{earthquakemod}

We consider an area in the vicinity of a subduction zone. Two tectonic plates are moving with respect to each other, raising progressively the mechanical constraints
applied between them. Assuming the boundary between the plates to be linear ($d=1$) or planar ($d=2$),  we consider a set $[0,D]^d$ and for every $t\geq 0$ a function $\tilde h_t:[0,D]^d\mapsto \R$ called profile at time $t$. This latter function records the intensity of the force 
applied at each point in  $[0,D]^d$ and at time $t$.  We denote by $(T_n)_{n\geq 0}$ the non-decreasing sequence recording the times at which the  earthquakes take place successively ($T_0=0$ by convention).  
Between two consecutive earthquakes, the motions of both plates increase the stress profile linearly in time 
and uniformly on $[0,D]^d$ with speed $C>0$. When an earthquake takes place, the constraints in $[0,D]^d$ are 
partially relaxed so that $\tilde h$ decreases discontinuously at $T_N$.
Thus, for $N\in \N$, we write for $\x \in [0,D]^d$
\begin{align}\label{evoltempfieconst}
\tilde h_{T_N+t}(\x)&=\tilde h_{T_N}(\x)+C\,t, \quad t\in [0,T_{N+1}-T_N)\\
\nonumber \tilde h_{T_{N+1}}(\x)&=\lim_{t\to T_{N+1}^-} \tilde h_{t}(\x)-\phi_{N+1}(\x),
\end{align} 
where $\phi_{N+1}:[0,D]^d\mapsto \R^+$ is a random function that describes the effect of the $(N+1)$-th earthquake on the profile.

In the present paper, we will only consider the evolution of the profile due to earthquakes. Thus, we get rid of the linear growth of the profile
induced by the constant $C$ in \eqref{evoltempfieconst} and only consider the profile immediately after each earthquake.
Moreover, we prefer working with positive modifications of the profile at earthquakes
and therefore we will consider $(h_N)_{N\geq 0}$ defined, for $n\geq 0$, as 
$$
h_N(\x)=-(\tilde h_{T_N}(\x)-C\, T_N), \quad \quad \x \in [0,D]^d,$$
rather than $\tilde h$. 
In other words,
we will consider a sequence $(h_N)_{N\geq 0} $ of random continuous functions on $[0,D]^d$ satisfying $h_0=0_{[0,D]^d}$ and for $N\geq 0$
\begin{equation}\label{defhcons}
h_{N+1}(\x)=h_N(\x)+\phi_{N+1}(\x),\quad \x\in [0,D]^d.
\end{equation}

We observe that the rand-model is a particular case of  \eqref{defhcons} for which (recall \eqref{elemprotrans}) the $N$-th profile transformation  
is given by 
\begin{equation}
\phi_{N}(\x):=X_N(\x), \quad \x\in [0,D]^d,
\end{equation}
so that the sequence of profile $h$ is the one introduced in \eqref{defrand} above.
To shed some lights on the choices made in (\ref{eq:psihol}--\ref{defrand}), we observe first that the characteristic length of an earthquake is observed by geophysicists  to be a heavy-tailed random variable. That is the reason why the sequence  $Z$ is chosen to be Pareto distributed and we note that $\beta$ is expected to be close to $3$ (see e.g. \cite{HK04} or \cite{KHHTK12}). The sequence $Y$ provides the centers of each 
earthquakes taken into account. In the rand-model, it is chosen to be i.i.d. and $Y_1$ follows the uniform distribution 
on $[0,D]^d$.  Finally, the function $\psi$ allows us to take into account the fact that the force modification due to an earthquake varies in space and decays as one gets away from its  epicenter.

\smallskip

%
\begin{remark} \rm
Note that, in the present paper, we will not consider the law of the inter-arrival times $(T_{N+1}-T_N)_{n\geq 1}$ between consecutive earthquakes. For the rand-model,
it seems reasonable to assume that $(T_{N+1}-T_N)_{n\geq 1}$ is an i.i.d sequence of random variables. The situation is much more involved for the min-model (introduced below)
and even though this is not the object of the present paper,  determining the law of $T_{N+1}-T_N$ for a given $n$ is a very interesting question. 
\end{remark}

We describe now a modification of the rand-model for which a  transformation is centered at the minimum of the 
profile. The physical relevance of such a modification comes from the belief  that an earthquake starts when the 
mechanical constraint applied at the boundary between plates reaches a threshold at a given point. We assume that between earthquakes, all the system experiences a linear in time force increase, and that the threshold of an earthquake initiation is the same at all points.  This yields that 
for $N\geq 0$ the $N+1$-th profile transformation is centered at one of the minimums of $h_N$.  This is the min-model. 


 Since the set of minimums is not necessarily reduced to a singleton, we
consider $U:=(U_i)_{i\geq 1}$ an i.i.d. sequence of random variables, independent of $Z$ and following the uniform law on $[0,D]^d$. Then, the center of the $N$-th earthquake is the 
closest point of $U_N$ in $[0,D]^2$  at which the minimum of $h_{N-1}$ is attained, i.e., 
\begin{equation}\label{centerm}
M_N:=\argmin\{v_{U_N}(\x)\colon h_{N-1}(\x)=\min_{[0,D]^d} h_{N-1}\}.
\end{equation}
Thus, for $N\geq 1$ 
\begin{equation}\label{defminmod}
h_{N+1}^{\text{m}}(\x)=h_N^{\text{m}}(\x)+  Z_{N+1}^{\alpha-d}\   \psi\Big(\frac{v_{M_{N+1}}(\x)}{Z_{N+1}}\Big),\quad \x\in [0,D]^d.
\end{equation}
With Figure \ref{fig:both_figuresconst}, we display an example of construction of $h_4$ and $h_4^{\text{m}}$ with the same values of transformation widths $Z_1,Z_2,Z_3,Z_4$.

\begin{remark} \rm
With the min-model, we take into account the situation where an earthquake is initiated as soon as the profile $\tilde h$ (recall \eqref{evoltempfieconst}) reaches a critical threshold $K>0$. Thus, 
\begin{equation}\label{deftn}
T_{N+1}=\inf \Big\{t>T_N\colon \max_{[0,D]^d} \tilde h_t=K \Big\}.
\end{equation}
However the growth of $\tilde h$ being constant and uniform on $[T_N,T_{N+1})$, the point at which the $(N+1)$-th earthquake will start is already 
known at time $T_N$ (that is after the profile transformation of the $N$-th earthquake has been applied). This is the reason why we drop the 
constant loading between two consecutive earthquakes and focus on the transformations of the profile induced by the earthquakes since they 
are sufficient to define the center-point of each earthquake.  
\end{remark}

\begin{figure}[h]
    \centering
    \begin{subfigure}[b]{0.45\textwidth}
        \centering
        \includegraphics[width=\textwidth]{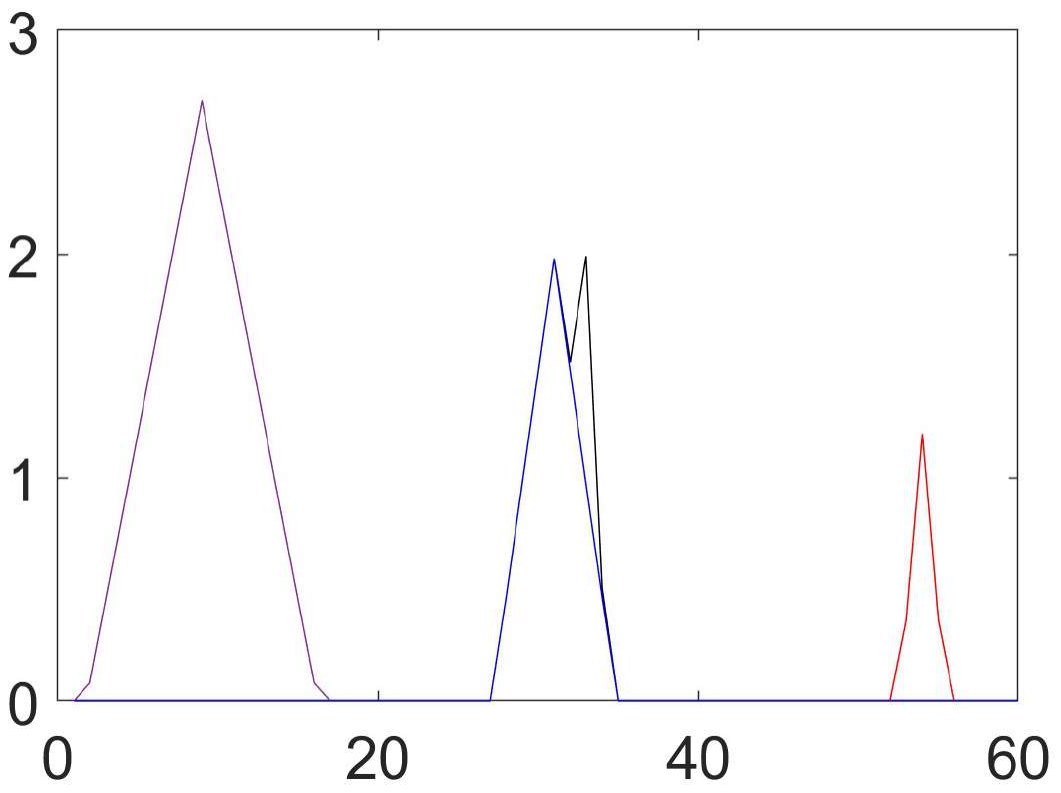}
        \vspace{-2cm}
        \label{fig:figure1}
    \end{subfigure}
    \hfil
    \begin{subfigure}[b]{0.45\textwidth}
        \centering
        \includegraphics[width=\textwidth]{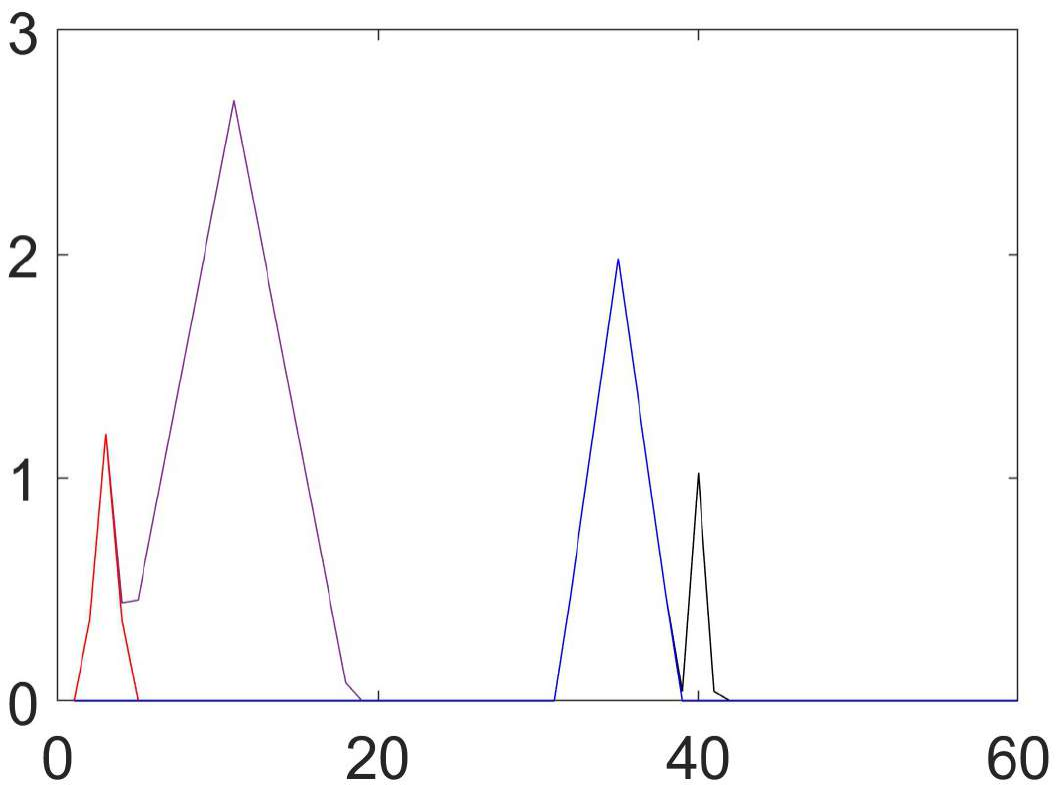}
         \vspace{-2cm}
        \label{fig:figure2}
    \end{subfigure}
    \caption{In dimension $d=1$  and for $D=60$, sampling of the rand model after $4$ transformations (i.e. $h_4$) on the left  and of the min model after $4$ transformations (i.e., $h_4^{\text{min}}$) on the right.  Note that we have used  the triangle function $\psi(x)=(1-x) \, \ind_{[0,1]}(x)$ and
   the same random variables  $Z_1,Z_2,Z_3,Z_4$ for both figures.  }
    \label{fig:both_figuresconst}
\end{figure}


\subsection{Relationship  to some different random deposition models}\label{rela}

\begin{enumerate} 
  \item The (BD) ballistic deposition model (see  \citet{Sep2000},
    \citet{Pen2008},  \citet{ComDalSag2023}). On each size $x \in \Z$
  rectangular blocks fall down at random with rate $1$, independently
  of other sites. Falling blocks have each width $1$ and their own
  random height, where the heights corresponding to different blocks
  are IID. This model is a Markov process on the space $\R^\Z$. It is
  a \emph{local} model, when we exclude overhangs : the rate of attaching new particles depends
  only on the states of some finite number of neighboring sites.
  \item The discrete polynuclear growth (PNG) is a local growth model
  where the height above $x$ at time $t$ satisfies
  \begin{equation}
    h(x,t+1) = \max\etp{h(x-1,t),h(x,t),h(x+1,t)} + \omega(x,t+1)
  \end{equation}
  with $\omega(x,t)$ IID. When the distribution of the $\omega(x,t)$
  is geometric with special constraints, the interface grows linearly
  with time and the transversal fluctuations
  are shown to be of the order $t^{2/3}$ (see \citet{Joh2003})
\end{enumerate}


\section{Results}\label{Resultat}

Theorem \ref{speedfront} below is concerned with the growth speed of the rand-model $h=(h_N)_{N=0}^\infty$.
It turns out that depending on $\alpha$ and $\beta$ the profile growth speed  is either ballistic or 
super-ballistic. Subsequently, we will state Theorem \ref{speedfrontvariant} which is the counterpart of Theorem \ref{speedfront} for the min-model and the stellar-model. Note that, in the super-ballistic regime, 
the min model displays the same profile growth speed and the same limit as the rand model. In the ballistic regime, in turn, we are able to prove the tightness of $(h_N^{\text{m}}/N)_{N\in \N}$ only and not its finite dimensional convergence. 
The fact that another sequence of centered random variables enters the definition of  the  stellar-model modifies the 
profile growth speed only at criticality.

The profile fluctuations are considered for the rand model in Theorem \ref{fluctconver} and for the stellar model in Theorem \ref{convstel}. We have much less results for the fluctuations of the min model (see Proposition \ref{tightkappalar1}). 
%
%
%
%
%
\subsection{Notations for distributional convergences}
We recall that  for $(\x,\y)\in \R^d$ the $1$-norm $|\x-\y|_1:=|\x_1-\y_1|+\dots+|\x_d-\y_d|$.
We will consider two types of convergence modes for the random functions that we consider in the present paper.  First the convergence for finite dimensional distributions. 
To that aim we let $(L_N)_{N\geq 1}$ a sequence of random functions from $(\Omega,\cA,P)$ to $\big((\R^k)^{[0,D]^d}, \text{Bor}\,(\R^k)^{\otimes [0,D]^d}\big)$.
We let $\mu$ be a probability on $\big((\R^k)^{[0,D]^d}, \text{Bor}\,(\R^k)^{\otimes [0,D]^d}\big)$. and $L_\infty$ a random function of law $\mu$.
\begin{definition}
The sequence $(L_N)_{N\in \N}$ converges for finite dimensional distributions towards $\mu$ if for every $j\in \N$ and every $(\s_1,\dots,\s_j)\in ([0,D]^d)^j$,  the sequence of real random vectors
$(L_N(\x_1),\dots,L_N(\x_j))_{N\in \N}$ converges in distribution
towards $(L_\infty(\x_1),\dots,L_\infty(\x_j))$. Then, we note

\begin{equation}\label{convfinitedim}
L_N \fddn \mu.
\end{equation}
This notation will be extended to 
\begin{equation}\label{limlawalt}
L_N\fddn  L_\infty 
\end{equation}
when the convergence occurs towards the law of $L_\infty$ a random continuous function.
\end{definition}

The second mode of convergence takes into account the continuity of the random functions. 
We denote by $C_d^k$ the set of real continuous functions defined on  $[0,D]^d$ and taking values in $\R^k$. The set $C_d^k$ is endowed with the 
uniform convergence norm denoted by $||\cdot||_{\infty}$ and we let $\cC_d^k$ be the associated Borel $\sigma$-algebra.
We let  $\mu$ be a probability law on $(C_d^k,\cC_d^k)$ and we let $(L_N)_{N\geq 1}$ be a sequence of random continuous functions from $(\Omega,\cA,P)$ to $(C_d^k,\cC_d^k)$.
\begin{definition}\label{limlawcont}
The sequence $(L_N)_{N\in \N}$ converges in law towards $\mu$ if for every \newline $F: (C_d^k,||\cdot||_{\infty})\mapsto (\R, |\cdot|\, )$ continuous and bounded, it holds true that
$$\lim_{N\to \infty} E\big[F(L_N)\big]=\int_{C_d^k} F(f) \ \mu (df).$$
Then, we note
\begin{equation}\label{limlaw}
L_N\Rightarrow_{N} \mu.
\end{equation} 
 This notation will be extended to 
\begin{equation}\label{limlaw}
L_N\Rightarrow_{N} L_\infty 
\end{equation}
when the convergence occurs towards the law of $L_\infty$ a random continuous function.

\end{definition}

For the sake of simplicity, we will drop the index $k$ from all notations when we consider real random functions, that is when $k=1$.

\subsection{Profile growth speed and fluctuations for the rand-model}\label{grwthspeed}

We recall (\ref{eq:psihol}--\ref{defrand}) and on the same probability space $(\Omega,\cA,P)$, we define   $\xi:=(\xi_{i})_{i\geq 1}$ an i.i.d. sequence of random variables following an exponential law of parameter $1$. We also set  $T_i:=\xi_1+\dots+\xi_i$ for $i\in \mathbb{N}$.
For $B$ a random vector, we let $B \, 1_{[0,D]^d}$ be the random function that equals $B(\omega)$ on $[0,D]^d$.

Those notations allow us to state our main Theorems for the rand-model. We note that, in Theorem \ref{speedfront}, the limiting processes only depend on the three parameters $\alpha,\beta$ and $d$. Moreover, the continuity of $\psi$ on $[0,1]$ is sufficient to perform the proof.

Two exponents will be of particular importance in the next two theorems, i.e.,
\begin{equation}\label{defkzeroandkappa}
\kzero=\frac{\alpha -d}{\beta -1} \quad \quad \text{and} \quad \quad   \kappa=\frac{\alpha-n-d}{\beta-1},
\end{equation}
since they provide the exponent of the growth speed of the profile  (see Theorem \ref{speedfront} below) and 
 of the fluctuations of the profile (see Theorem \ref{fluctconver}) when the tail of the $Z$ random variables is very large. 
   We will see that both $\kzero$ and $\kappa$
appear as well in the counterparts theorems for the min model and the stellar model.

With Figure \ref{phasediag}, we provide the phase diagram of the rand model. We distinguish between $4$ areas of interest depending on the relative positions of $\kzero$ and $\kappa$ with respect to $1$ and $1/2$ respectively.


\begin{theorem}\label{speedfront}
 Assume $\psi \in \cH^0$.
For the next two cases,
the limits are constant random variables, whose respective values are
also constant functions on $[0,D]^d$ . 
 
\begin{enumerate}
\item If $\kzero>1$, then
\begin{equation}\label{convgN}
N^{-\kzero} h_N \Rightarrow_{N} \Bigg[\sum_{i=1}^\infty \frac{1}{T_i^{\kzero}}\Bigg] \ \ind_{[0,D]^d}.
\end{equation}

\item If $\kzero=1$, then
\begin{equation}\label{convgN3}
(N \log(N))^{-1} h_N \Rightarrow_{N} \ind_{[0,D]^d}.
\end{equation}

\item If $\kzero<1$, then
\begin{equation}\label{convgN2}
N^{-1} h_N \Rightarrow_{N} \gamma_{\alpha,\beta,D} \,   \ind_{[0,D]^d}.\\
\end{equation}

An analytic expression of $\gamma_{\alpha,\beta,D}$ is
$$
\gamma_{\alpha,\beta,D}:=\esp{\psi\Big(\frac{v_{Y_1}(0)}{Z_1}\Big) Z_1^{\alpha
    -d}} =\,  \frac{\beta-1}{D^d} \int_{[1,\infty)\times [0,D]^d} 
z^{\alpha-d-\beta} \ \psi\Big(\frac{v_{\y}(0)}{z}\Big) \, dz\, dy.
$$

%

\end{enumerate}
\end{theorem}
\smallskip


\begin{remark}
{\rm
  When $\kzero>1$, for fixed $s$, $h_N(s)$ is a sum of
  i.i.d. random variables with heavy tail, therefore classical limit
    theorems imply the convergence to a stable random variable of
    index  $\unsur{\kzero}$.  What Theorem
    \ref{speedfront} implies  is that we have joint convergence,
    for different times $s_1, s_2, \ldots s_k$ of the random variables
    $(h_n(s_i))_{i\le k}$ to a $k$-uple $(A,...,A)$ of the same random
    variable $A= \sum_{i\ge 1} T_i^{- \kzero}$. Indeed, we prove the convergence in distribution of the
    whole process.

The limiting random variable $A$ has  characteristic function 
\begin{equation}\label{carcfunlim}
\Phi(t)=\esp{e^{it A}}=\exp\Big[-\Gamma\big(1-\tfrac{1}{\kzero}\big)\,  |t|^{\unsur{\kzero}}\, \big(\cos(\tfrac{\pi}{2 \kzero})-\, i \, \text{sign}(t) \, \sin(\tfrac{\pi}{2\, \kzero})\big)\Big].
\end{equation}
The fact that the series $\sum_{i\geq 1} T_i^{-\kzero}$ is a stable $\frac{1}{\kzero}$ random variable is well known in the probabilist folklore and we refer to \cite[Theorem 16.25]{K13} for a description of the stable laws and  to \cite[Equation (1.1.6) and Theorem 1.4.5]{SamTaq94} for the characteristic function in \eqref{carcfunlim}. 
}

\end{remark}

\begin{remark}[Heuristics of the critical point]
\rm If $\kzero<1$,  the random variable $Z_1^{\alpha-d}$ is integrable, which allows us to apply the law of large number to 
$(h_N(\x))_{N\in \N}$ for every $\x\in [0,D]$. This explains why $(h_N)_{N\in \N}$ has to be renormalized by $N$ to converge in distribution 
towards a non-trivial limit.
\end{remark}

The next theorem is concerned with the growth speed of the fluctuations of the front around its 
value at $0$. In other words we study the limit in distribution of the sequences of random processes 
$(f_N)_{N\geq 1}$ defined in \eqref{deffn}. We consider $d\in \{1,2\}$ since these are the physically relevant
dimensions in our case.

%

\begin{theorem}\label{fluctconver}
  Assume $\psi \in \cH^n$ and pick $d\in \N$.
  

\begin{enumerate}\label{mainres}
\item If $\kappa > \undemi$ and $d\in \{1,2\}$, then
\begin{equation}\label{premconv}
N^{-\kappa} f_N \Rightarrow_{N} \mu_d,
\end{equation}
where $\mu_d$ is a probability law on $(C_d,\cC_d)$ that will be defined properly in Section \ref{descrmu} below.


%

\item 
If $d=1$ and $\kappa < \undemi$, then
\begin{equation}\label{secconv}
N^{-\frac{1}{2}} f_N \Rightarrow_N Y
\end{equation}
where $Y$ is a centered Gaussian process of covariance 
$r(s,t)=\text{Cov}\,(X_1(s),X_1(t))$.

\item If $d\in\ens{1,2}$ and $\kappa = \undemi$ then
\begin{equation}
  (N \log N)^{-\frac{1}{2}} f_N \Rightarrow_N Y
\end{equation}
where $Y$ is a centered Gaussian field of covariance
\begin{equation}
  r(s,x)=\text{Cov}\,(A(s),A(x))\quad\text{with } A(s) =\frac{\psi^{(n)}(0)}{n!}(
  v_U(s)^n -v_U(0)^n)\,,
\end{equation}
and $U$ is a uniform on $[0,D]^d$.
\item 
If $d=2$ and $ \kappa<1/2$, then 
\begin{equation}\label{terconv}
N^{-\frac{1}{2}} f_N  \fddn Y,
\end{equation}
where $Y$ is a centered Gaussian field of covariance 
$r(\s,\x)=\text{Cov}\,(X_1(\s),X_1(\x))$.

\item 
If $d=2$ and $\kappa<1/4$, then 
\begin{equation}\label{terconv}
N^{-\frac{1}{2}} f_N \Rightarrow_N Y,
\end{equation}
where $Y$ is a centered Gaussian field of covariance 
$r(\s,\x)=\text{Cov}\,(X_1(\s),X_1(\x))$.


\end{enumerate}

\vspace{1cm}

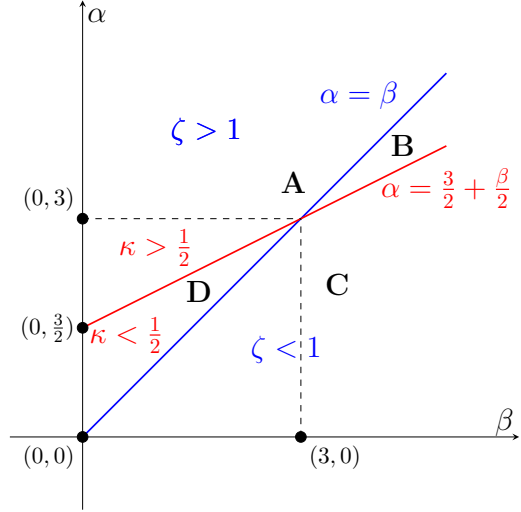
\begin{figure}[h]
    \centering
    \begin{tikzpicture}[scale=0.8]
        \begin{axis}[
            axis lines=middle,
            xmin=-1, xmax=6,
            ymin=-1, ymax=6,
            xtick=\empty,
            ytick=\empty,
            width=10cm,
            height=10cm,
            domain=0:5,
            samples=100,
            axis equal image
        ]
        \addplot[blue, thick] {x};
        \node[blue, below] at (axis cs:3.8,5) {\Large $\alpha = \beta$};

        \addplot[red, thick] {1.5 + 0.5*x};
        \node[red, below] at (axis cs:5,3.8) {\Large $\alpha = \frac{3}{2} + \frac{\beta}{2}$};

        \addplot[only marks, mark=*, mark size=2.5pt, black] coordinates {(0,0)};
        \node[below left] at (axis cs:0,0) {$(0,0)$};

        \addplot[only marks, mark=*, mark size=2.5pt, black] coordinates {(0,3)};
        \node[above left] at (axis cs:0,3) {$(0,3)$};

        \addplot[only marks, mark=*, mark size=2.5pt, black] coordinates {(3,0)};
        \node[below right] at (axis cs:3,0) {$(3,0)$};

        \addplot[only marks, mark=*, mark size=2.5pt, black] coordinates {(0,1.5)};
        \node[left] at (axis cs:0,1.5) {$(0,\frac{3}{2})$};

        \draw[dashed] (axis cs:0,3) -- (axis cs:3,3);
        \draw[dashed] (axis cs:3,0) -- (axis cs:3,3);

        \node at (axis cs:2.9,3.5) {\Large \textbf{A}};
        \node at (axis cs:0.2,5.8) {\Large $\alpha$};
        \node at (axis cs:5.8,0.2) {\Large $\beta$};
        \node at (axis cs:1.7,4.2) {{\color{blue} \Large $\kzero>1$}};
         \node at (axis cs:1.7,4.2) {{\color{blue} \Large $\kzero>1$}};
          \node at (axis cs:1,2.6) {{\color{red} \Large $\kappa>\frac{1}{2}$}};
           \node at (axis cs:.6,1.4) {{\color{red} \Large $\kappa<\frac{1}{2}$}};
           \node at (axis cs:2.8,1.2) {{\color{blue} \Large $\kzero<1$}};
        \node at (axis cs:4.4,4) {\Large \textbf{B}};
        \node at (axis cs:3.5,2.1) {\Large \textbf{C}};
        \node at (axis cs:1.6,2) {\Large \textbf{D}};
        \end{axis}
    \end{tikzpicture}
      \caption{Phase diagram of the rand model in dimension $d=1$ for a function $\psi \in  \cH^1$.
     Note that in zones A and D, the profile growth is super-ballistic whereas it is simply ballistic in zones B and C. The fluctuations, in turn are non Gaussian 
     in zones A and B whereas they are Gaussian in zones C and D.}
     \label{phasediag}
\end{figure}

\vspace{0.5cm}

\end{theorem}

\begin{remark}
{\rm
Since $\kappa<\kzero$, we have shown that the typical size of the fluctuations of the 
sequence of random profiles $(\frac{1}{N} h_N)_{N\in \N}$ is much smaller than the typical size of 
the profiles themselves.}
\end{remark}

\smallskip
\medskip

%
%
%
%
%
%

\begin{remark}[Explanation of the $\kappa_c=\undemi$ critical value in dimension $1$]  \label{explibetac}
{\rm
 Let us now give a heuristic explanation of the fact that, for the random model,  the growth rate of the fluctuations of the front (as a function of $N$) 
changes dramatically at $\kappa_c$. In order to apply the Central Limit Theorem  to $f_N$
and obtain a diffusive behavior in  $\sqrt{N}$ it is sufficient that  
$X_1(\s)-X_1(0)$ (recall \eqref{deffn}) is in $L^2$. Observe that,
thanks to Lemma~\ref{lem:hng}, since $v_y(x) \le dD$,
\begin{equation}\label{boundzx}
  \valabs{X_1(\s)-X_1(0)} = Z_1^{\alpha -d}
  \valabs{\psi\Big(\frac{v_{Y_1}(\s)}{Z_1}\Big) -
    \psi\Big(\frac{v_{Y_1}(0)}{Z_1}\Big)} 
  \le C' Z_1^{\alpha -(n+d)}
\end{equation}
A simple calculation  shows that
\begin{equation}\label{L2forz} 
  Z_1^{\alpha -(n+d)} \in L^2 \iff \kappa < \undemi.
\end{equation}}
\end{remark}

\begin{remark}[Explanation of the $\kappa_c=1/4$ critical value in dimension 2]
{\rm
In dimension $2$, the proof of the convergence for finite dimensional marginals of $(N^{-\frac{1}{2}} f_N)_{N\in \N}$
is similar to that in dimension $1$ since it only requires for $Z_1^{\alpha-n-d}$ to be in $L^2$. Therefore, 
it suffices that $\kappa<1/2$ to check this finite dimensional convergence. Proving the 
 the tightness of $(N^{-\frac{1}{2}} f_N)_{N\in \N}$ is more difficult since Proposition \ref{kolmog} can not be applied with $p=2$ and $d=2$.
What we can do, in turn is applying this proposition with $p=4$, $d=2$ and $\gamma=2$, but for this we need 
$\kappa<1/4$. }
\end{remark}

\smallskip

\subsubsection{Description of the limiting laws $\mu_1$ and $\mu_2$ and $\mu_{\text{stel}}$}\label{descrmu}


{\rm For $d\in \{1,2\}$, we characterize the law $\mu_d$ in Theorem \ref{mainres}. Note that $\mu_1$ is the law of a   random continuous function whereas $\mu_2$ is that of a random  continuous surface.
 Thus, for all $i\geq 1$ we set
\begin{equation}\label{defGi}
G_i(\x):=\frac{\psi^{(n)}(0)}{n!} \big(v_{Y_i}(\x)^n-v_{Y_i}(0)^n\big), \quad \x\in [0,D]^d.
\end{equation}
With Lemma \ref{L2} we show that $T_i^{-\kappa}\in L^4(\Omega,\cA,P) $ if   $i>4 \kappa$.  
For this reason, we let $n_1=\lfloor 4 \kappa\rfloor$ and 
for every $N\geq n_1$, we define on $(\Omega,\cA,P)$ a continuous random process $\gamma_{1}$ and a sequence of random processes $(\gamma_{2,N})_{N> n_1}$
as 
\begin{align}\label{defgamma12}
\gamma_{1}(\x):&=\sum_{i=1}^{n_1} \frac{G_i(\x)}{T_i^\kappa}, \quad\x\in [0,D]^d, \\
\nonumber \gamma_{2,N}(\x):&=\sum_{i=n_1+1}^N \frac{G_i(\x)}{T_i^\kappa} ,\quad\x\in [0,D]^d,
\end{align}
 so that  for all $\x\in [0,D]^d$ and $N>n_1$ we have $\gamma_{2,N}(\x)\in L^4(\Omega,\cA,P)$. 
With proposition \ref{convgamma2N} below, we prove that 
for all $\x\in [0,D]^d$,  the sequence of real random variables $(\gamma_{2,N}(\x))_{N\geq 1}$ converges $P$-almost surely towards a  random variable $\gamma_{2,\infty}(\x)$ that belongs to $L^2(\Omega,\cA,P)$.

\begin{definition}\label{defmud}
This allows us to define $\mu_d$ as a probability law on $\big(\R^{[0,D]^d}, \text{Bor}\,(\R)^{\otimes [0,D]^d}\big)$ whose finite dimensional marginals are, for $k\in \N$ and  $\s_1,\dots,\s_k\in [0,D]^d$ 
given by the law of the random vector
\begin{equation}\label{defgamma}
(\gamma_1(\s_1)+\gamma_{2,\infty}(\s_1),\dots,\gamma_1(\s_k)+\gamma_{2,\infty}(\s_k)).
\end{equation}
\end{definition}
In fact, Lemma \ref{lem:tkol1} establishes the tightness of $(\frac{1}{N^{\kappa}} f_N)_{N\in \N}$ in 
$(C_d,\cC_d)$. Therefore, $\mu_d$ can be considered as a probability law on $(C_d,\cC_d)$.



\begin{proposition}\label{convgamma2N}
For $\kappa> 1/2$  and for all $\x\in [0,D]^d$,  the sequence of real random variables $(\gamma_{2,N}(\x))_{N\geq n_1+1}$ converges almost surely and in $L^2(\Omega,\cA,P)$ towards a random variable $\gamma_{2,\infty}(\x)$. 
\end{proposition}
The proof of this proposition is postponed to Appendix \ref{proof-of-convgamma2N}.

\begin{remark}
{\rm
 In the case  $\kappa >1$, instead of $\kappa>\undemi$, the limiting measure $\mu_d$ is simpler to describe. 
 It is indeed the law of the random continuous function
 \begin{align}\label{defmupart}
\s\in [0,D]^d\mapsto \sum_{i=1}^\infty \frac{G_i(\s)}{T_i^\kappa}.
\end{align}
At this stage,  one can verify that $P$-almost surely, the infinite sum of functions in the right hand side of \ref{defmupart} converges normally.
This comes from the fact that there exists a $c>0$ such that $||G_i||_{\infty}\leq c$ for all $i\in \N$ and $\omega\in \Omega$.  Thus, 
\begin{align}
\sum_{i=1}^\infty \frac{||G_i||_{\infty}}{T_i^\kappa}&\leq c \sum_{i=1}^\infty \frac{1}{T_i^\kappa},
\end{align}
and it remains, since $\kappa >1$, to apply the law of large numbers that guarantees us
that $P$-almost surely $T_i=i(1+o(1))$ when  $i\to \infty$.}
 
 \end{remark}
 
 It remains to define $\mu_{\text{stel}}$. To that aim we modify \eqref{defgamma12} into 
 \begin{align}\label{defgamma12}
\gamma_{1}^{\text{stel}}(\x):&=\sum_{i=1}^{n_1} \frac{G_i(\x) \Theta_i}{T_i^\kappa}, \quad\x\in [0,D]^2, \\
\nonumber \gamma_{2,N}^{\text{stel}}(\x):&=\sum_{i=n_1+1}^N \frac{G_i(\x) \Theta_i}{T_i^\kappa},\quad\x\in [0,D]^2.
\end{align}
For every $\x\in [0,D]^2$, both $\gamma_{1}^{\text{stel}}(\x)$ and $\gamma_{2,N}^{\text{stel}}(\x)$ are random vectors of dimension $3$. 
Without any substantial modification, we can repeat the proof of Proposition \ref{convgamma2N} to show that each $3$ coordinates of $(\gamma_{2,N}^{\text{stel}}(\x))_{N\geq 1}$ converges $P$-almost surely
as $N\to \infty$. As a consequence the sequence  $(\gamma_{2,N}^{\text{stel}}(\x))_{N\geq 1}$ converges $P$-almost surely and we denote by $\gamma_{2,\infty}^{\text{stel}}(\x)$ its limit.

\begin{definition}\label{defmustel}
We define $\mu_{\text{stel}}$ as a probability law on $\big((\R^3)^{[0,D]^2}, \text{Bor}\,(\R^3)^{\otimes [0,D]^2}\big)$ whose finite dimensional marginals are, for $k\in \N$ and  $\s_1,\dots,\s_k\in [0,D]^2$ 
given by the law of the random vector
\begin{equation}\label{defgamma}
(\gamma_1^{\text{stel}}(\s_1)+\gamma_{2,\infty}^{\text{stel}}(\s_1),\dots,\gamma_1^{\text{stel}}(\s_k)+\gamma_{2,\infty}^{\text{stel}}(\s_k)).
\end{equation}
\end{definition}
By mimicking the proof of  Lemma \ref{lem:tkol1} for each three coordinates in the stellar model framework, we obtain that $\mu_{\text{stel}}$ can be considered as a probability law on $(C_2^3,\cC_2^3)$.

\bigskip

\subsection{Convergence for the min-model and for the stellar radiations model}\label{grwthspeedseis}
In the present section we provide the counterparts of Theorems \ref{speedfront} for the min-model and for the stellar model. We also provide the counterpart of Theorem \ref{fluctconver} 
for the stellar model. The reason why we are not able for the moment to provide convergence results for the fluctuations of the min-model  
 is that its growth process  is non-local. The fact that the subset of $[0,D]^d$ on top of which the profile  $h_N^{\text{m}}$ grows is located in the vicinity of one of the global minimums of $h_N^{\text{m}}$ introduces a non-trivial dependency between $h_N^{\text{m}}$ and the increment $h_{N+1}^{\text{m}}-h_N^{\text{m}}$. This complicates significantly the study.  Recall \eqref{defkzeroandkappa}.

\begin{theorem}\label{speedfrontvariant}
 Assume $\psi \in \cH^0$ and let $\kzero=\frac{\alpha -d}{\beta -1}$.
For the next two cases,
the limits are constant random variables, whose respective values are
also constant functions on $[0,D]^d$. 
 
\begin{enumerate}
\item If $\kzero>1$, then
\begin{equation}\label{convgNmin}
N^{-\kzero} h_N^{\text{m}} \Rightarrow_{N} \Bigg[\sum_{i=1}^\infty \frac{1}{T_i^{\kzero}}\Bigg] \ \ind_{[0,D]^d}, 
\end{equation}
and
\begin{equation}\label{convellN}
N^{-\kzero} h^{\text{stel}}_N \Rightarrow_{N} \Bigg[ \sum_{i=1}^\infty \frac{\Theta_i}{T_i^{\kzero}}\Bigg] \ \ind_{[0,D]^d}.
\end{equation}

\item If $\kzero=1$, then
\begin{equation}\label{convgN3min}
(N \log(N))^{-1} h_N^m \Rightarrow_{N} \ind_{[0,D]^d},
\end{equation}
and 
\begin{equation}\label{convellNal}
N^{-1} h^{\text{stel}}_N \Rightarrow_{N} \sum_{i=1}^\infty \frac{\Theta_i}{T_i^{\kzero}} \ \ind_{[0,D]^d}.
\end{equation}


\item If $\kzero<1$, then
\begin{equation}\label{convgNstel}
N^{-1} h^{\text{stel}}_N \Rightarrow_{N} 0.\\
\end{equation}

\item
If $\kzero<1$ then $(N^{-1} h^m_N)_{N\in \N}$ is a tight sequence in $C_d,||\cdot||_{\infty}$.


%

\end{enumerate}
\end{theorem}

\medskip

\begin{remark}
For now, we are not able to prove the convergence in finite dimensional distribution of $(N^{-1} h^m_N)_{N\in \N}$ towards $\gamma_{\alpha,\beta,D} 1_{[0,D]^d}$ as it is the case for the rand-model in \eqref{convgN2}.
 However,  since the function $s\to v_Y(\s)=dist(\s-\y, D
\Z^d)$ is $D$-periodic for all coordinates, its integral is the same for every cube of
size $D$, i.e.,
\begin{align*}
 \int_{[0,D]^d}
  \psi\etp{\frac{v_{\x}(M_k)}{Z_k}} \, d\x=  \int_{[0,D]^d}\psi\etp{\frac{v_{\x}(0  )}{Z_k}} \, d\x 
\end{align*}
so that after a straightforward application of the law of large numbers, we obtain that for $P$-a.e. $\omega \in \Omega$
\begin{equation}\label{convmean}
\lim_{N\to \infty} \int_{[0,D]^d} N^{-1} h^m_N(x) dx= D^d \gamma_{\alpha,\beta,D}  
\end{equation}
\end{remark}

As far as fluctuations are concerned, the sole result that we are able to display is the tightness in the Proposition below. The proof is postponed to Section \ref{tightkappalar11}.
\begin{proposition}\label{tightkappalar1}
If $\kappa>1$, the sequence $(\frac{1}{N^\kappa} f_N^{\text{m}})_{N\in \N}$ is tight in $\cC_d,||\cdot||_{\infty}$.
\end{proposition}

As mentioned above, for the next Theorem we consider the stellar model only. Thus, $d=2$ and we use $\mu_{\text{stel}}$
a probability measure  on $(C_2^3,\cC_2^3)$ that will be described in definition \ref{defmustel} above. We recall that $(\Theta_i)_{i\in \N}$
is an i.i.d. sequence of centered random vectors. We write $\Theta_i=(\Theta_i^1,\Theta_i^2,\Theta_i^3)$ its coordinates.

For notational convenience, each $k$ dimensional  vector $V=(v_1,\dots,v_k)$ is viewed as an element in $\mathcal M_{k,1}(\R)$, that is a line vector.
A column vector of the same dimension is denoted by $V^{\intercal}$. For $(V,U)\in M_{k,1}(\R)^2$, the vectorial product $V U^{\intercal}$  will sometimes be
denoted by $\langle V,U \rangle$. 

\begin{theorem}\label{fluctconverstel}
  Assume $\psi \in \cH^n$ and let
  $\kappa=\frac{\alpha-n-2}{\beta-1}$.
  

\begin{enumerate}\label{convstel}
\item For $\kappa>1/2$

\begin{equation}\label{premconvevect}
N^{-\kappa} f^{\text{stel}}_N \Rightarrow_{N} \mu_{\text{stel}},
\end{equation}
where $\mu_{\text{stel}}$ is a probability law on $(C_2^3,\cC_2^3)$ that will be defined properly in Section \ref{descrmu} below.

\item 
If $\kappa<1/2$, then 
\begin{equation}\label{terconv}
N^{-\frac{1}{2}} f_N^{\text{stel}}  \fddn Y,
\end{equation}

where $Y:[0,D]^2\mapsto \R^3$ is a centered Gaussian field with covariance function
$$r(\s,\x)=E\big[X_1(\s) X_1(\x)\big]\,  E\big[\Theta_1^{\intercal}\, \Theta_1\big]\in \mathcal{M}_3(\R).$$

\item 
If $\kappa<1/4$, then the preceding finite dimensional convergence can be extended to a
convergence for the whole process
\begin{equation}\label{terconvevect}
N^{-\frac{1}{2}} f_N^{\text{stel}} \Rightarrow_N Y\,.
\end{equation}
\item If $\kappa=\undemi$, then 
\begin{equation}
  (N \log N)^{-\frac{1}{2}} f_N \Rightarrow_N Y
\end{equation}
where $Y$ is a centered Gaussian field of covariance
\begin{equation}
  r(s,x)=\text{Cov}\,(A(s),A(x))\esp{\Theta_1^{\intercal}\, \Theta_1}\quad\text{with } A(s) =\frac{\psi^{(n)}(0)}{n!}(
  v_U(s)^n -v_U(0)^n)\,,
\end{equation}
and $U$ is a uniform on $[0,D]^d$.  
  
\end{enumerate}

\end{theorem}

\subsection{Perspectives and discussion.}\label{perspec}

Completing our investigation of the min model requires to identify the asymptotics of the fluctuations for $(h_N^{\text{m}})_{N\in \N}$.
This is more difficult than it is for the rand and stellar models. Indeed, the fact that for every $i\in \N$ the center position $Y_i$ 
of the $i$-th transformation depends on the whole trajectory $h_{i-1}^{\text{m}}$  makes the computations to obtain the tightness of $(h_N^{\text{m}})_{N\in \N}$ and its convergence in finite dimensional distribution much more intricate. However, we can 
make some conjectures.
\begin{enumerate}
\item In the case where $\kappa<\kappa_c$, we expect that for $d\in \{1,2\}$
the result obtained in Theorem \ref{fluctconver}  for the rand model holds true for the min model as well. Indeed, in this case 
the tail of the $Z$ variables is so heavy that only finitely many of them (the largest of course) are sufficient to provide a good 
approximation of the 
limiting distributions (cf. the definitions of $\mu_1$ and $\mu_2$ in Section \ref{descrmu} and Figure \ref{fig:both_figureskappapet}).  The indices of those very big variables $Z$ are 
chosen uniformly at random among $\{1,\dots,N\}$ and therefore the space between them also tends to $\infty$ with $N$ (at least in probability).
Assuming some mixing features of the growth model we can reasonably expect that the dependence between the $Y$ variables associated with those 
very large values of $Z$ vanishes as $N\to \infty$. Therefore, we should recover the limits $\mu_1$ for $d=1$ and $\mu_2$ for $d=2$.

\item When $\kappa>\kappa_c$, the tail of the $Z$ variables is lighter. 
We observe with Figure \ref{fig:both_figureskappapetit} that the fluctuations seem to be of finite order but we are not able to provide convincing heuristics for the moment.

\end{enumerate}     

\begin{figure}[h]
    \centering
    \begin{subfigure}[b]{0.45\textwidth}
        \centering
        \includegraphics[width=\textwidth]{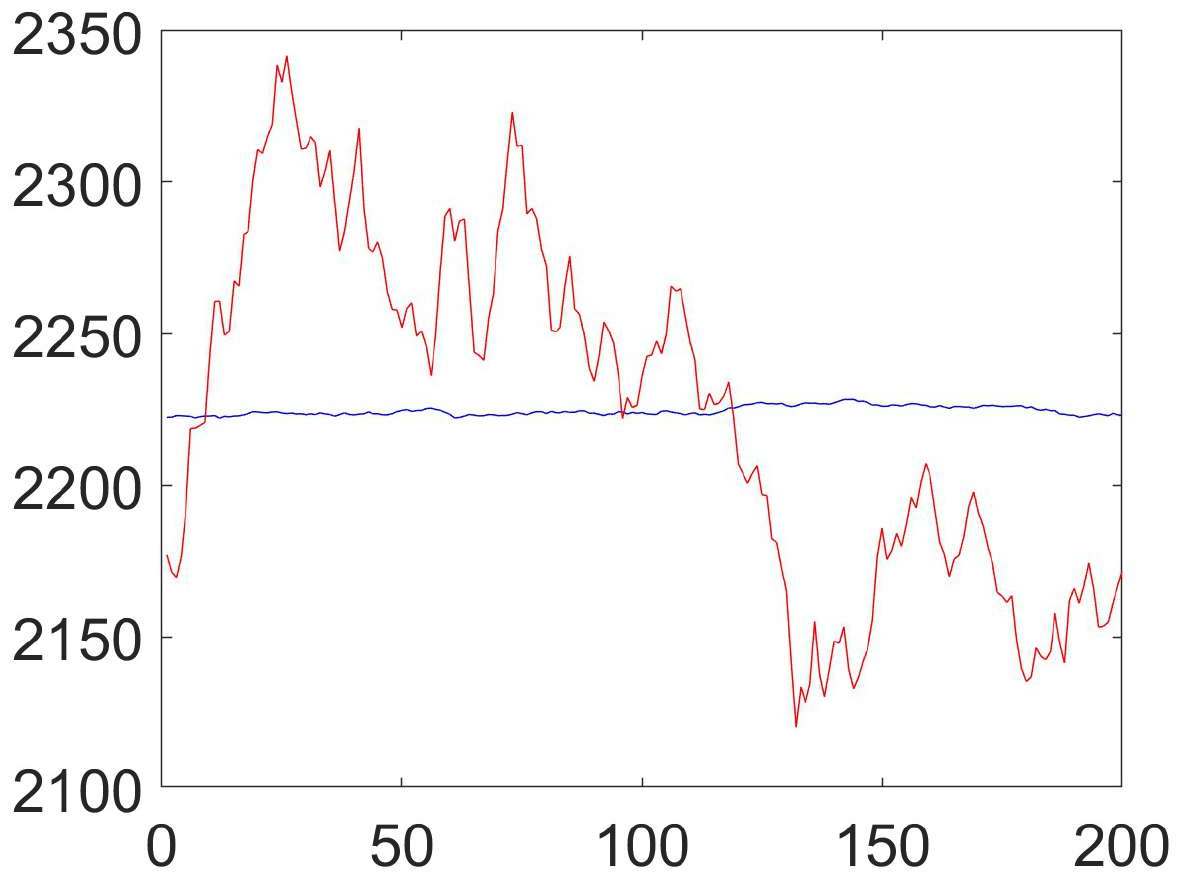}
        \vspace{-2cm}
        \caption{Sampling of $h_N$ for the rand model (in orange) and for the min model (in blue), with 
        $d=1$, $N=10^4$, $D=200$, $\alpha=1.5$, $\beta=2$ and $\psi(x)=(1-x) \, \ind_{[0,1]}(x)$ that is $n=1$.
        Note that the sequence of random variables $(Z_i)_{i=1}^N$ is the same for both sampling. In this case $\kzero=0.5$ and $\kappa=-0.5$. Therefore, the rand model is ballistic with diffusive fluctuations and the min model is ballistic as well.}
        \label{fig:figure1}
    \end{subfigure}
    \hfill
    \begin{subfigure}[b]{0.45\textwidth}
        \centering
        \includegraphics[width=\textwidth]{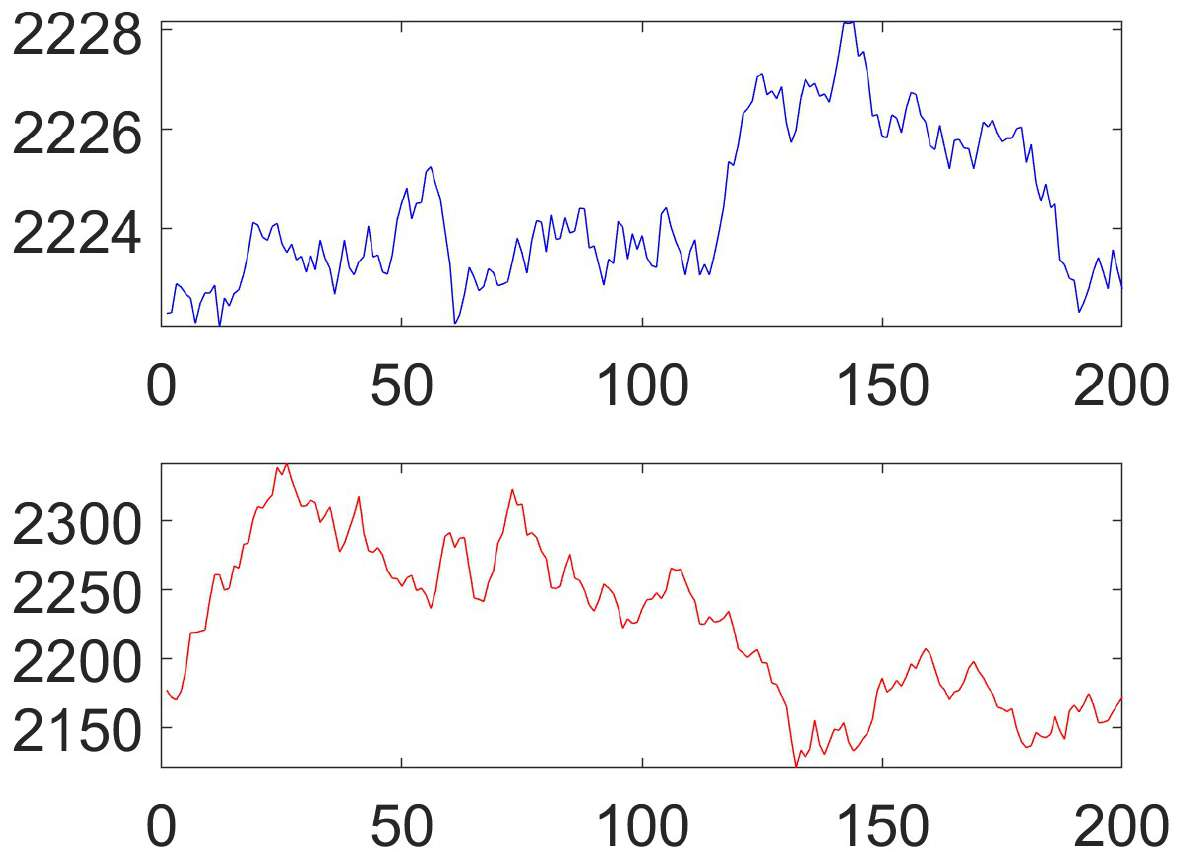}
         \vspace{-1.2cm}
        \caption{The lower figure is a zoom of $h_N$ for the random model drawn on the left. We observe that the 
        fluctuations are of order $100$ which corroborates the result \eqref{secconv} in Theorem \ref{fluctconver}. The upper figure is a zoom of $h_N^{\text{m}}$ for the min model 
        drawn on the left. The fluctuations seem to be of finite order.}
        \label{fig:both_figureskappapetit}
    \end{subfigure}
    \caption{}
\end{figure}

\begin{figure}[h]
    \centering
    \begin{subfigure}[b]{0.48\textwidth}
        \centering
        \includegraphics[width=\textwidth]{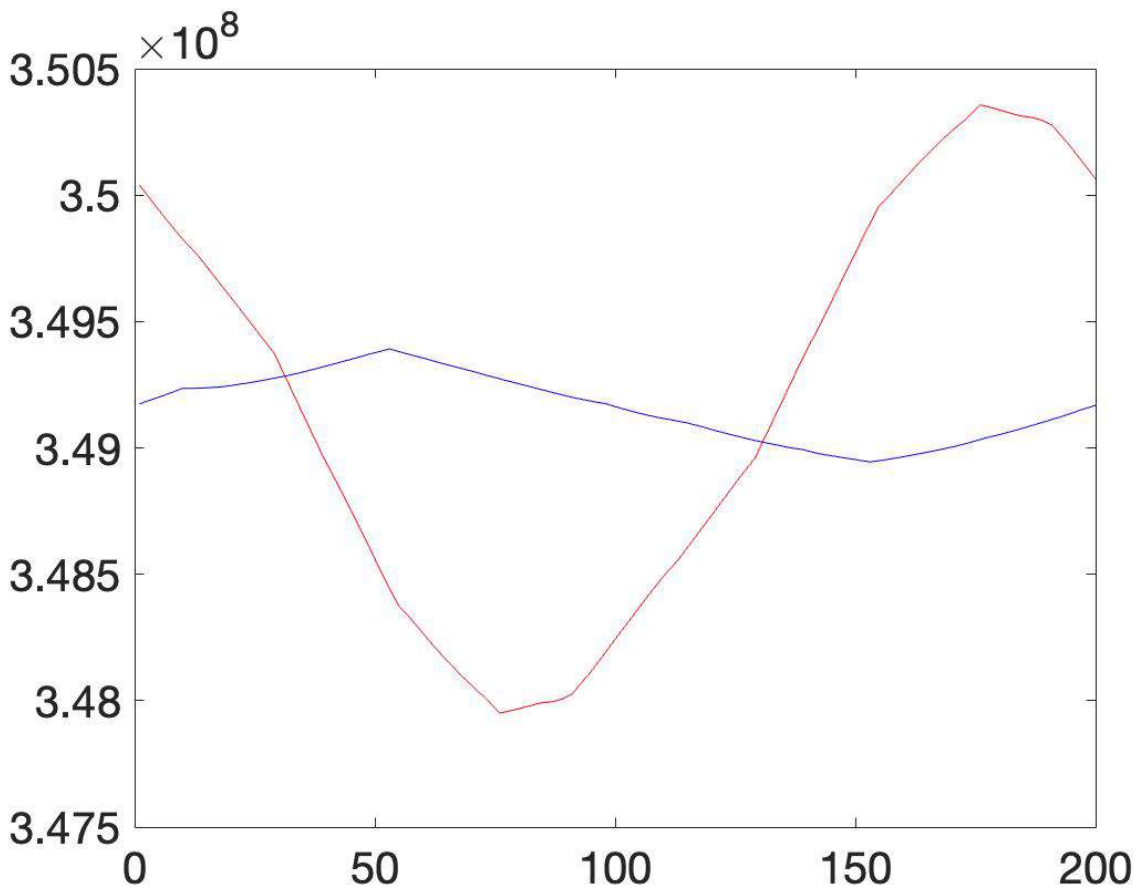}
        \vspace{-3cm}
        \caption{Sampling of $h_N$ for the rand model (in orange) and for the min model (in blue), with 
        $d=1$, $N=10^4$, $D=200$, $\alpha=3$, $\beta=2$ and $\psi(x)=(1-x) \, \ind_{[0,1]}(x)$ that is $n=1$.
        Note that the sequence of random variables $(Z_i)_{i=1}^N$ is the same for both sampling. In this case $\kappa=1$.}
        \label{fig:figure1}
    \end{subfigure}
    \hfill
    \begin{subfigure}[b]{0.48\textwidth}
        \centering
        \includegraphics[width=\textwidth]{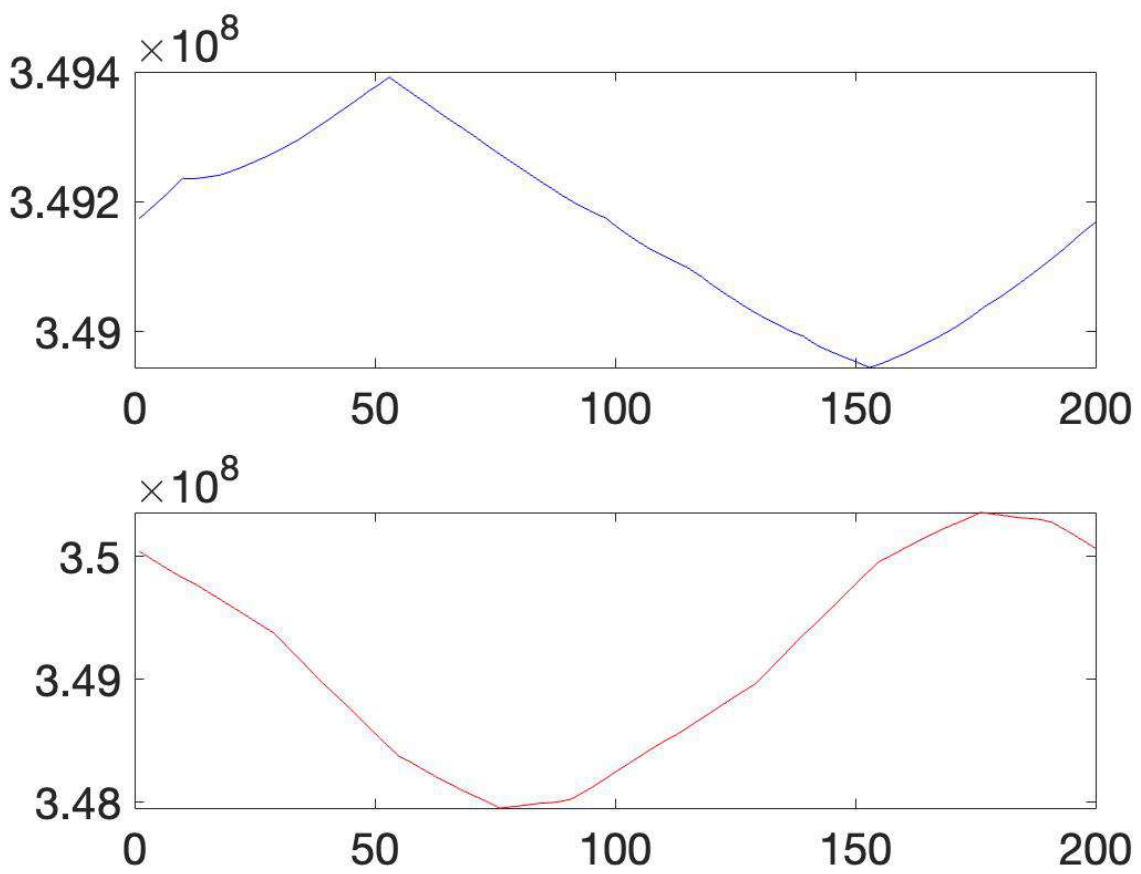}
         \vspace{-3cm}
        \caption{The lower figure is a zoom of $h_N$ for the rand model drawn on the left and the upper figure is a zoom of $h_N^{\text{m}}$ for the min model. We observe that in this case, the global shape of both curves representing $h_N$ and $h_N^{\text{m}}$ seems to be imposed by finitely many transformations (those associated with the largest values of $Z$).}
        \label{fig:figure2}
    \end{subfigure}
    \caption{}
    \label{fig:both_figureskappapet}
\end{figure}

\section{Preparation: convergence in distribution of random functions in $(C_d,\cC_d)$.}\label{tool-for-distr-conv}

In this Section we describe some strategies to prove the convergence in distribution of a sequence
of random functions. 

\subsection{Case of real random functions.}\label{realrandfunc} 

We consider $(L_N)_{n\in \N}$ a sequence of random variables from $(\Omega,\cA,P)$ to $(C_d,\cC_d)$. 
In other words, for every $N\in \N$ we let  $L_N(\omega): [0,D]^d\mapsto \R$ be a continuous function
such that $L_N:(\Omega,\cA,P)\mapsto (C_d,\cC_d)$ is a random variable. We will need several times in this paper to 
prove that  $(L_N)_{N\in \N}$ converges in law towards $\mu$ a probability measure on $(C_d,\cC_d)$. The following 
allows us to prove such convergences.


\begin{theorem}{\cite[Theorem 7.1]{B}}\label{convlawpro}
 The sequence $(L_N)_{N\in \N}$ converges in law towards $\mu$ if and only if 
$(L_N)_{N\in \N}$ is a tight sequence of random variables
and  $(L_N)_{N\in \N}$ converges to $\mu$ for finite dimensional marginals, i.e., 
 for $k\in \N$ and $(\s_1,\dots,\s_k) \in ([0,D]^d)^k$,
 \begin{equation}\label{margfi}
 \lim_{N\to \infty} (L_N(\s_1),\dots,L_N(\s_k))=_{law} (S(\s_1),\dots,S(\s_k))
 \end{equation}
 with $S$ a random function of law $\mu$.
\end{theorem}

\begin{remark}\label{remconvprobponct}
{\rm
The convergence in probability of $(L_N(\s))_{N\in \N}$ to $S(\s)$ for every $\s\in [0,D]^d$ implies the 
convergence of finite dimensional marginals \eqref{margfi}.}
\end{remark}

Proving the tightness of $(L_N)_{N\in \N}$  requires to define the continuity modulus of function in $C_d$. To that aim, for $h\in C_d^k$ and $\delta>0$ we set 
\begin{equation}\label{contmod}
 w_{h}(\delta):=\sup_{\s, \x\in [0,D]^d\colon\; |\s-\x|_1 \leq \delta}\quad  |h(\s)-h(\x)|.
 \end{equation}
 
\begin{proposition}{\cite[Theorem 7.3]{B}}\label{tightness1}
If  $(L_N(0))_{N\in \N}$ is a tight  sequence of real random variables and if
for every $\gep >0$
\begin{equation}\label{tension} 
\lim_{\delta \to 0} \, \limsup_{N\to \infty}  P\Big(w_{L_N}(\delta)\geq \gep \Big)=0,
\end{equation}
then,  $(L_N)_{N\in \N}$ is tight sequence of random functions.
\end{proposition}

The tightness of $(L_N)_{N\in \N}$ may also be proven with the help of Kolmogorov criterion 
(see e.g.\cite{B}) and stated in the following Proposition. This proposition, although well-known 
by probabilists is proven in Appendix \ref{kolmogcritproof}.
\begin{proposition}\label{kolmog}
Let $d\in \N$ and let $(L_N)_{N\in\N}$ be a sequence of real valued random continuous functions on $[0,D]^d$.
\begin{itemize}
\item  If there exist  $p>1,C>0$ and $\gamma>0$ such that   for all $\x,\s\in [0,D]^d$
\begin{equation}\label{tensionkol}
\underset{N\geq 1}{\sup} \E\Big[\Big| L_N(\x)-L_N(\s)\Big |^p\Big]\leq C |\x-\s|_1^{d+\gamma},
\end{equation}
\item if the sequence of real random variables $(L_N(0))_{N\in \N}$
is tight,
\end{itemize}
then $(L_N)_{N\in \N}$ is a tight sequence of random continuous functions.

\end{proposition}\label{mathtools}

\subsection{Case of vector  random functions}\label{vectconv}

In the framework of the stellar model, we need to take into account sequences of random functions $(L_N)_{N\in \N}$
such that $L_N:[0,D]^2\mapsto \R^3$. We let  $L_N^{1}, L_N^2, L_N^3$ be the coordinates of $L_N$. First, it turns out that checking the tightness of each coordinate
is sufficient to derive the tightness of $(L_N)_{N\in \N}$.
\begin{proposition}\label{moredimconv}
If  $(L^j_N)_{N\in \N}$ is a tight sequence of real random functions for every $j\in \{1,2,3\}$,
then $(L_N)_{N\in \N}$ is also a tight sequence of random functions. 
\end{proposition}
Proving this proposition is not difficult, it suffices to extend the definition of the continuity modulus in \eqref{contmod} and to observe that for every $h:=(h_1,h_2,h_3)\in C_d^3$,
$$w_h(\delta)\leq w_{h_1}(\delta)+w_{h_2}(\delta)+w_{h_3}(\delta).$$

The convergence of finite dimensional marginals is defined as in \eqref{margfi}. Remark \eqref{remconvprobponct} 
may be restated here since it suffices that for every $j\in \{1,2,3\}$ and every $\s\in [0,D]^2$ the sequence of real random variables
$(L_N^j(\s))_{N\in \N}$ converges in probability to  $S^j(\s)$ to conclude that $(L_N)_{N\in \N}$ converges for finite dimensional marginals
to $S$.

%


\section{Proof of Theorems  \ref{speedfront} and \ref{speedfrontvariant}}\label{Preuves1}
We will first display, for the rand-model the proofs of \eqref{convgN}, \eqref{convgN3}  \eqref{convgN2} and for the min-model that of \eqref{convgNmin}, \eqref{convgN3min} and of the tightness in point 4 of Theorem 
\ref{speedfrontvariant}.  Then, we will consider the stellar-model and prove \eqref{convellN}, \eqref{convellNal} and \eqref{convgNstel} by explaining where the proofs diplayed for the rand-model  have to be modified.

\subsection{Common proofs for the rand-model and the min-model}

For the case $\kzero>1$ and $\kzero=1$ the proofs for the rand-model and the min-model are exactly the same. For this reason, in these two cases, 
we assume that $(Y_i)_{i\geq 1}$ is a given sequence of random variables 
defined on $(\Omega,\cA,P)$ and taking values in $[0,D]^d$ without any assumption on its law.

\subsubsection{Case $\kzero>1$} 
\begin{remark}\label{remstatord}
{\rm
For $N\in \N$, we let $(Z^{(1)},\dots,Z^{(N)})$ be the order statistics of 
$(Z_1,\dots,Z_N)$, i.e.,   
\begin{equation}\label{statordre}
Z^{(1)}>Z^{(2)}>\dots>Z^{(N)}.
\end{equation}
For simplicity, we have omitted the $N$ dependency of each random variable in \eqref{statordre}. 
At this stage, we consider  $\xi:=(\xi_{i})_{i\geq 1}$ an i.i.d. sequence of random variables following an exponential law of parameter $1$.
For every $i\geq 1$ we set
$T_i:=\xi_1+\dots+\xi_i$. We use the equality in distribution   
\begin{equation}\label{statordrepalt}
(Z^{(1)},\dots,Z^{(N)})=_{\text{law}}\Bigg[\bigg(\frac{T_{N+1}}{T_1}\bigg)^{\frac{1}{\beta-1}},\dots,\bigg(\frac{T_{N+1}}{T_N}\bigg)^{\frac{1}{\beta-1}}\Bigg].
\end{equation}}
\end{remark}

Because of \eqref{statordrepalt}
 we observe that $N^{-\kzero}\,  h_N$ has the same law in $(C,\cC)$ as $\big[\frac{T_{N+1}}{N}\big]^{\kzero}\, \tilde h_N$ with
\begin{equation}\label{defgN2}
 \tilde h_N(\s):= \sum_{i=1}^N \Big[ \frac{1}{T_i}\Big]^{\kzero}\  \psi\bigg(\frac{v_{Y_i}(\s)}{Z_N^{(i)}}\bigg),\quad \s\in  [0,D]^d
\end{equation}
where for simplicity we have denoted by $(Z_N^{(1)},\dots,Z_N^{(N)})$ the random variables in the r.h.s. of \eqref{statordrepalt}.
Thus, we are left with proving that
\begin{equation}\label{convgN22}
\Big[\frac{T_{N+1}}{N}\Big]^{\kzero}\, \tilde h_N \Rightarrow_{N} \Bigg[\sum_{i=1}^\infty \frac{1}{T_i^{\kzero}}\Bigg]\  \ind_{[0,D]^d}.
\end{equation} 
The law of large numbers guarantees us that 
\begin{equation}\label{llnu}
\lim_{N\to \infty} \frac{T_{N+1}}{N}=1 \quad \text{for} \quad P\text{-a.e.}\quad  \omega \in \Omega,
\end{equation}
and therefore, we can drop 
the term $\big[\frac{T_{N+1}}{N}\big]^{\kzero}$ in the l.h.s. of \eqref{convgN22}. 
 At this stage, we can claim that \eqref{convgN} will be proven once we show that  
for $P$-a.e. $\omega\in \Omega$,
\begin{equation}\label{convuniffront}
\lim_{N\to \infty} \Big |\Big|  \tilde h_N-\sum_{i=1}^{\infty}\Big[\frac{1}{T_i}\Big]^{\kzero}\Big |\Big|_{\infty}=0.
\end{equation} 
We write $\tilde h_N= \bar h_N+\hat h_N$ where $\bar h_N$ (resp. $\hat h_N$) is the sum over the indices $i\in \{1,\dots, N/\log N\}$ (resp.   $i\in \{1+N/\log N,\dots, N\}$) in \eqref{defgN2}.
Consequently, \eqref{convuniffront} will be proven once we show that for $P$-a.e. $\omega\in \Omega$,
\begin{align}\label{frontconv}
&\lim_{N\to \infty} \Big |\Big|  \bar h_N-\sum_{i=1}^{\frac{N}{\log N}}\Big[\frac{1}{T_i}\Big]^{\kzero}\Big |\Big|_{\infty}=0\\
\nonumber &\lim_{N\to \infty} \sum_{i=1+\frac{N}{\log N}}^{ \infty}\Big[\frac{1}{T_i}\Big]^{\kzero}=0\\
\nonumber &\lim_{N\to \infty} ||  \hat h_N||_{\infty}=0
\end{align} 
The third convergence in \eqref{frontconv} is a straightforward consequence of the second convergence since $\psi$ is bounded to the extend that it is  continuous on $[0,\infty)$ and equals $0$ on $[1,\infty)$. 
 By \eqref{llnu}, for $P$-a.e  $\omega\in \Omega$ we can claim that  $T_i=i \,  (1+o(1))$. 
Moreover,  $\kzero>1$ and therefore $\sum_{i\geq 1} 1/(T_i)^{\kzero}$ converges for $P$-a.e. $\omega\in \Omega$
which gives the second convergence in \eqref{frontconv}. It remains to prove the  first convergence in \eqref{frontconv}. To that aim we recall that $v_{\y}(\s)$ is bounded above by $dD/2$ (see e.g. Lemma \ref{lem:loivys}) and we observe that 
\begin{equation}\label{conv3concl}
 \Big |\Big|  \bar h_N-\sum_{i=1}^{ \frac{N}{\log N}}\Big[\frac{1}{T_i}\Big]^{\kzero}\Big |\Big|_{\infty}\leq  \sup_{u\in \big[0,\frac{d D}{2}\big]}
 \bigg|\psi\bigg(u \bigg[\frac{T_{\frac{N}{\log N}}}{T_{N+1}}\bigg]^{\kzero}\bigg)-1\bigg|\quad  \sum_{i=1}^{\infty} \Big[\frac{1}{T_i}\Big]^{\kzero}.
\end{equation}
We conclude that the r.h.s. in \eqref{conv3concl} converges for $P$-a.e. $\omega\in \Omega$ towards $0$ as $N\to \infty$ since the sum in the r.h.s. is almost surely finite, $\psi$ is continuous at $0$ and for $P$-a.e. $\omega\in \Omega$ the convergence 
$\lim_{N\to \infty}T_{\frac{N}{\log N}}/T_{N+1}=0$ holds true by \eqref{llnu}. This completes the proof for the rand-model and the min-model of Theorem \ref{speedfront} in the case $\kzero>1$.

\subsubsection{Case $\kzero=1$}
We use $\tilde h_N$ again in \eqref{defgN2} so that the proof consists in showing that 
\begin{equation}\label{convgN222}
\frac{1}{\log N}\  \tilde h_N \Rightarrow_{N} \ind_{[0,D]^d}.
\end{equation} 
At this stage, we note that by \eqref{llnu} and by equivalence of positive diverging sums, it holds that for $P$-a.e. $\omega\in \Omega$, 
\begin{equation}\label{equivserielog}
\sum_{i=1}^k \frac{1}{T_i}=\log k\,  (1+o(1))
\end{equation}
so that the proof of that case will be complete once we show that for $P$-a.e. $\omega\in \Omega$,
\begin{equation}\label{convuniffront2}
\lim_{N\to \infty} \frac{1}{\log N}\Big |\Big|  \tilde h_N-\sum_{i=1}^{N} \frac{1}{T_i}\Big |\Big|_{\infty}=0.
\end{equation} 
To prove \eqref{convuniffront2}, we adapt \eqref{frontconv} to the present framework, that is the same three convergences  
have to be proven except that they are divided by $\log N$ and that the summation in the seconde equality is restricted to 
$i\in \{1+N/\log N,\dots, N\}$. The proof of the first convergence  is a consequence of  \eqref{conv3concl} except that the summation in the r.h.s. 
is bounded because it is divided by $\log N$. The second and the third convergences are consequences of \eqref{equivserielog} since for $P$-a.e. $\omega\in \Omega$
$$\sum_{i=1+\frac{N}{\log N}}^{N}\ \frac{1}{T_i}=o(\log N),$$
and of the boundedness of $\psi$ for the third convergence.
This completes the proof of the case $\kzero=1$ for the rand-model and the min-model.

\subsection{Case $\kzero<1$ }
In this case, we prove \eqref{convgN2} that is the convergence in distribution of $(N^{-1} h_N)_{N\in \N}$ for the rand-model. For the min-model, we only prove the tension 
of $(N^{-1} h^m_N)_{N\in \N}$. Thus, the first step below covers both the rand and the min model whereas the second step is for the rand-model only.

 We will apply the strategy displayed in Section \ref{tool-for-distr-conv}. To that aim we write
\begin{equation}\label{defgN3}
 h_N(\s):= \sum_{i=1}^N Z_i^{\alpha-d} \psi\bigg(\frac{v_{Y_i}(\s)}{Z_i}\bigg),\quad \s\in [0,D]^d.
\end{equation}
The first step consists in proving that the sequence of random continuous processes $(\frac{1}{N}\, h_N)_{N\geq 1}$ is tight in $(C_d,\cC_d)$. To that aim, we will prove that 
\begin{equation}\label{tension} 
\lim_{\delta \to 0} \, \lim_{N\to \infty}  P\Big( \frac{1}{N} w_{h_N}(\delta)\geq \gep\Big)=0
\end{equation}
where, for $h\in C_{[0,D]^d}$ and $\delta>0$, 
$$ w_{h}(\delta):=\sup_{\s, \x\in [0,D]^d\colon\; |\s-\x|_1 \leq \delta}\quad  |h(\s)-h(\x)|.$$
We set $G:[1,\infty)\times [0, dD/2]\mapsto \R$ as $G(z,u)= \psi\big(\frac{u}{z}\big)$. We recall that 
$\psi$ is continuous on $[0,\infty)$, that  $\psi(0)=1$ and that $\psi(v)=0$ for $v\in [1,\infty)$. Thus, $G$ is continuous and 
$$\lim_{|z|\to \infty}\ \sup_{u\in [0,dD/2]} \big | G(z,u)-1|=0.$$
As a consequence, $G$ is uniformly continuous on $[1,\infty)\times [0,d D/2]$. Moreover, $\forall \y,\s, \x\in [0,D]^d$ it holds that 
$|v_{\y}(\x)-v_{\y}(\s)|\leq |\x-\s|_1$ so that if $|\x-\s|_1\leq \delta$
\begin{align}
\big |h_N(\s)-h_N(\x)\big| &\leq \sum_{i=1}^{N} Z_i^{\alpha-d}\ \big | G(Z_i, v_{Y_i}(\s))-G(Z_i, v_{Y_i}(\x)) \big |\\
&\leq \sum_{i=1}^{N} Z_i^{\alpha-d}\ w_G(\delta) .
\end{align}  
Consequently, for all $\gep >0$;
\begin{align}\label{derlitens}
P\Big[\frac{1}{N}  w_{h_N}(\delta)\geq \gep\Big]\leq P\Big[ w_G(\delta)  \frac{1}{N} \sum_{i=1}^N Z_i^{\alpha-1}\geq \gep\Big] .
\end{align}  
Since $Z_1^{\alpha-d}\in L^1$, the law of large number guarantees us that for 
$P$-a.e. $\omega\in \Omega$ it holds that  $ \lim_{N\to \infty} \frac{1}{N} \sum_{i=1}^N Z_i^{\alpha-1}=E(Z_1^{\alpha-1})$.
Thus, with the help of the uniform continuity of $G$ we can pick $\delta$ such that $w_G(\delta) E(Z_1^{\alpha-1})<\gep$
which is sufficient to obtain \eqref{tension}.

It remains to check that $(\frac{1}{N} h_N(0))_{N\geq 1}$ is tensed as a sequence of real random variables. To that aim we pick $\eta>0$  and we write
for $a>0$
\begin{equation}\label{tensionsec}
P\Big[\frac1N h_N(0)\geq a\Big]\leq P\Big[\frac1N \sum_{i=1}^N Z_i^{\alpha-1}  ||\psi||_{\infty}\geq a\Big].
\end{equation}
Using the law of large number again, it suffices to chose $a>E(Z_1^{\alpha-1})\, ||\psi ||_\infty$ to make sure that the r.h.s. in \eqref{tension}
vanished as $N\to \infty$. As a consequence, one can choose $a$ large enough in such a way that the l.h.s. in \eqref{tension} is bounded above by $\eta$
uniformly in $N\in \N$. This completes the proof of the first step.

For the second step of the proof, we consider only the rand-model and therefore $(Y_i)_{i\in \N}$ is i.i.d. and $Y_1$ is uniform on $[0,D]^d$.  We pick $k\in \N$ and 
$(\s_1,\dots,\s_k)\in ([0,D]^d)^k$. We recall \eqref{defgN3} and we note that Lemma \ref{lem:loivys} guarantees us that for every $\s\in [0,D]^d$
the sequence of random variables $\big(Z_i^{\alpha-d}
\psi\big(\frac{v_{Y_i}(\s)}{Z_i}\big)\big)_{i\in \N}$ is i.i.d.,
integrable and its law does not depend on $\s$. Consequently, the
random vector $(h_N(\s_1),\dots,h_N(\s_k))_{N\in \N}$ converges for
$P$-a.e. $\omega\in \Omega$ towards the constant vector whose
coordinates all equal $\gamma_{\alpha,\beta,D}$. This is sufficient to
complete the second step and the proof. 

\subsection{Proofs for the stellar model}

For $\kzero>1$, the proof of \eqref{convellN} is similar to that of \eqref{convgN} for the rand-model. The only difference is that, for every $i\in \N$, the random variable $X_i(\s)$ is multiplied by  the $3$-dimensional random vector
$\Theta_i$.  Thus, $\tilde h_N$ becomes
$$\sum_{i=1}^N \Theta_i \Big[ \frac{1}{T_i}\Big]^{\kzero}\  \psi\bigg(\frac{v_{Y_i}(\s)}{Z_N^{(i)}}\bigg),\quad \s\in  [0,D]^2.$$
Moreover, by definition (recall \eqref{spher}), we can claim that there exists a $c>0$ such that $P(|\Theta_1|_1\leq c)=1$ and $E(\Theta_1)=0$.  Then,  we replace $\sum_{i=1}^\infty 1/T_i^\kzero$ by  $\sum_{i=1}^\infty \Theta_i/T_i^\kzero$ and the proof 
(\ref{defgN2}--\ref{conv3concl}) is the same. 

For $\kzero<1$, the proof for the rand-model can also be adapted without difficulty to the stellar case. The function $h_N$ is modified into
$$\sum_{i=1}^{N} \Theta_i\, Z_i^{\alpha-d}\,  \psi\bigg(\frac{v_{Y_i}(\s)}{Z_i}\bigg).$$
The proof of the tightness of  $(h_N)_{N\in \N}$
is straightforwardly adapted from (\ref{defgN3}--\ref{tensionsec}). Finally, the convergence of the finite dimensional marginals is 
obtained by applying the law of large numbers to the i.i.d. sequence of centered random vectors 
$\big(\Theta_i\,  Z_i^{\alpha-d} \psi\big(\frac{v_{Y_i}(\s)}{Z_i}\big)\big)_{i\in \N}$ which is integrable since $Z_1^{\alpha-d}\in L^1$ when $\xi<1$
and centered since $E(\Theta_1)=0$.

Finally, the case that requires some attention is $\kzero=1$, i.e., \eqref{convellNal}.
Recall that for $j\in \N$, $\Theta_j:=(\Theta_j^1,\Theta_j^2,\Theta_j^3)$ is a 3-dimensional vector whose coordinates are centered and bounded.
We 
recall \eqref{defgN2} and modify it into 
\begin{equation}\label{deftilhn}
 \tilde h_N(\s):= \sum_{i=1}^N  \frac{\Theta_i }{T_i}\  \psi\bigg(\frac{v_{Y_i}(\s)}{Z_N^{(i)}}\bigg),\quad \s\in  [0,D]^2.
 \end{equation}
 We will obtain \eqref{convellNal} as follows. First, we will prove the convergence of the finite dimensional marginals of $(\tilde h_N)_{N\in \N}$ and then, we will prove its tightness. 
 
 We observe also that for $j\in \{1,2,3\}$, the sequence  $(u^{\,j}_N)_{N\in \N}:=(\sum_{i=3}^N \Theta^j_i/T_i)_{N\in \N}$ is a martingale bounded in $L^2$ by a constant time $\sum_{i=3}^\infty E(T_i^{-2})$ which is finite by Lemma \ref{L2}. 
As a consequence we denote by $\sum_{i=3}^\infty \Theta^j_i/T_i$ the almost sure limit of $(u^j_N)_{N\geq 3}$.  This also implies that the sequence of random vectors $(\sum_{i=1}^N \Theta_i/T_i)_{N\in \N}$ converges almost surely 
and we denote by  $\sum_{i=1}^\infty \Theta_i/T_i$ its limit.

   For the convergence of the finite dimensional marginals of $(\tilde h_N)_{N\in \N}$, according to Remark \ref{remconvprobponct}, it is sufficient to prove that, for every  $\s\in [0,D]^2$, the sequence $(\tilde h_N(\s))_{N\in \N}$ converges in probability towards 
$\sum_{i=1}^\infty \Theta_i/T_i$.
It will be sufficient to obtain the following convergence in probability:
\begin{equation}\label{limfinitedimmarg}
\lim_{N\to \infty} \tilde h_N(\s)-\sum_{i=1}^{\frac{N}{\log N}} \frac{\Theta_i}{T_i}=0_3.
\end{equation}
Since it is obvious that $P$-almost surely
\begin{equation}\label{limfinitedimmarginal}
\lim_{N\to \infty} \sum_{i=1}^2 \frac{\Theta_i }{T_i}\  \psi\bigg(\frac{v_{Y_i}(\s)}{Z_N^{(i)}}\bigg)-\sum_{i=1}^{2} \frac{\Theta_i}{T_i}=0_3,
\end{equation}
the proof will be complete once we show that
for  $j\in \{1,2,3\}$, the   random sequences 
$\big( \bar h_N^j(\s)-\sum_{i=3}^{N/\log N} \Theta_i^j/T_i)_{N\in \N}$ and $\big( \hat h_N^j(\s)\big)_{N\in \N}$ converge in $L^2$-norm  towards $0$ with
\begin{align}\label{defqunbart}
\bar h_N^j(\s)&=\sum_{i=3}^{\frac{N}{\log N}} \frac{\Theta^j_i }{T_i}\  \psi\bigg(\frac{v_{Y_i}( \s)}{Z_N^{(i)}}\bigg),\\
\nonumber \hat h_N^j(\s)&=\sum_{i=\frac{N}{\log N}+1}^{N} \frac{\Theta^j_i }{T_i}\  \psi\bigg(\frac{v_{Y_i}( \s)}{Z_N^{(i)}}\bigg).
\end{align}
We observe that 
\begin{equation}\label{estimsupl2}
v_N:=\bar h_N^j(\s)-\sum_{i=3}^{\frac{N}{\log N}}  \frac{\Theta_i^j}{T_i} =\sum_{i=3}^{\frac{N}{\log N}} \frac{\Theta_i^j }{T_i} \bigg[\psi\bigg(\frac{v_{Y_i}(\s)}{Z_N^{(i)}}\bigg)-1\bigg],
\end{equation}
and since $\Theta_1^j$ is centered we write 
\begin{equation}\label{upperboundL2iprocandliml}
||v_N||_{L^2}^2:=\sum_{i=3}^{\frac{N}{\log N}} E\big[(\Theta_1^j)^2\big]\ E\bigg[ \frac{1}{T_i^2} \bigg[\psi\bigg(\frac{v_{Y_i}(\s)}{Z_N^{(i)}}\bigg)-1\bigg]^2 \bigg],
\end{equation}
and 
\begin{equation}\label{hnorm2}
||\hat h_N^j(\s))||_{L^2}^2:=\sum_{i=1+\frac{N}{\log N}}^N E\big[(\Theta^j_1)^2\big]\ E\bigg[ \frac{1}{T_i^2} \bigg[\psi\bigg(\frac{v_{Y_i}(\s)}{Z_N^{(i)}}\bigg)\bigg]^2 \bigg].
\end{equation}
The convergence towards $0$ in \eqref{hnorm2} is a straightforward consequence of \eqref{equivgamma} combined with the fact that $E\big[(\Theta^j_1)^2\big]<\infty$ and that 
$\psi$ is bounded. For every $i\geq 3$, we recall that for $P$-almost every $\omega\in \Omega$ we have  $\lim_{N\to \infty} Z_N^{(i)}\geq \lim_{N\to \infty}  Z_N^{(N/\log N)}=\infty$ and we apply the dominated convergence theorem to assert that 
$$
\lim_{N\to \infty} E\bigg[ \frac{1}{T_i^2} \bigg[\psi\bigg(\frac{v_{Y_i}(\s)}{Z_N^{(i)}}\bigg)-1\bigg]^2 \bigg]=0$$
and it suffices to use that $\sum_{i\geq 3} E(1/T_i^2)<\infty$ to apply the dominated convergence again and 
prove that \eqref{upperboundL2iprocandliml} converges towards $0$ as well. This completes the proof of the convergence 
for finite dimensional marginals.

It remains to explore the tightness of $(\tilde h_N)_{N\in \N}$. It suffices to check that its first coordinate is 
tight, that is we consider the formula in \eqref{deftilhn} with $\Theta_i^1$ instead of $\Theta_i$. We note, with the help of \eqref{formexainteti} 
that $1/T_i \in L^4$ for $i\geq 5$. We split $\tilde h_N$ into 
\begin{equation}
\tilde h_N(\x)=a_N(\x)+b_N(\x):=\sum_{i=1}^4 \frac{\Theta_i^1}{T_i} \ \psi\bigg(\frac{v_{Y_i}(\x)}{Z_N^{(i)}}\bigg)+
\sum_{i=5}^N \frac{\Theta_i^1}{T_i} \ \psi\bigg(\frac{v_{Y_i}(\x)}{Z_N^{(i)}}\bigg), \x\in [0,D]^2.
\end{equation}
The fact that $(a_N)_{N\in \N}$ is a tight sequence of random functions is straightforward. Therefore, we focus on $(b_N)_{N\in \N}$ and we use Proposition \ref{kolmog}
with $p=4$ and $\gamma=2$. We recall that  for every $i\in \{1,\dots,N\}$ the lower bound $Z_N^{(i)}\geq 1$ holds true and also that $\forall \y,\s, \x\in [0,D]^3$ it holds that 
$|v_{\y}(\x)-v_{\y}(\s)|\leq |\x-\s|_1$. With the help of \ref{control diffpsi}  we may write that there exists $C>0$ such that for every
$\x,\s\in [0,D]^2$,
\begin{align}\label{prooftightkolmog}
\nonumber\E\Big[\big| b_N(\x)-b_N(\s) \big|^4\Big]&= \E\bigg[\bigg( \sum_{i=5}^N \frac{\Theta_i^1}{T_i} \Big[\psi\Big(\frac{v_{Y_i}(\x)}{Z_N^{(i)}}\Big)-\psi\Big(\frac{v_{Y_i}(\s)}{Z_N^{(i)}}\Big)\Big]\bigg)^4\bigg]\\
\nonumber &\leq C\, \E\big[(\Theta_1^1)^4\big]\,  |\x-\s|_1^4\,  \sum_{i=5}^\infty\E\bigg[ \frac{1}{T_i^4} \bigg]+12 C\,  |\x-\s|_1^4\ \sum_{5<i<j\leq N} \E\bigg[ \frac{(\Theta_i^1)^2 (\Theta_j^1)^2}{T_i^2 T_j^2} \bigg]\\
& \leq C |\x-\s|_1^4\ \E\big[(\Theta_1^1)^4\big]  \bigg[   \sum_{i=5}^\infty\E\bigg[ \frac{1}{T_i^4}\bigg]+ 12 \, \bigg(\sum_{i=5}^\infty\E\bigg[ \frac{1}{T_i^4}\bigg]^{1/2}\bigg)^2 \bigg]
\end{align}
where we have used that $(\Theta^1_i)_{i\geq 1}$ is  independent of $(T_i)_{i\in \N}$ and the Cauchy-Shwartz inequality. It remain to 
use Lemma \ref{L2} to assert that both infinite sums in the r.h.s. in \eqref{prooftightkolmog} are finite. This completes the proof.

%

\section{Proof of Theorem \ref{mainres} }\label{Preuves}
\begin{remark}
{\rm
We recall Remark \ref{remstatord} and in particular \eqref{statordre}. 
We
set $Y^{(1)}, \dots ,Y^{(N)}$ the random permutation of $(Y_1,\dots,Y_N)$ such that for every $(i,j)\in \{1,\dots,N\}^2$ we have  $Y^{(i)}=Y_j$ if and only if $Z^{(i)}=Z_j$. It is easily checked that
 $(Y^{(1)}, \dots ,Y^{(N)})$ is an i.i.d. family of random variables following the Uniform law on $[0,D]^d$ and remaining independent of the order statistics $(Z^{(1)},Z^{(2)},\dots,Z^{(N)})$.
Thus, in order to ease the notations and since we are only looking for convergence in law, we will use $(Y_1,\dots,Y_N)$ rather than
$(Y^{(1)}, \dots ,Y^{(N)})$. Note also that the same subtlety arises when working with the stellar model. Thus, we can keep working with  $(\Theta_1,\dots,\Theta_N)$ an i.i.d. family of  random vectors independent of both $Z$ and $Y$
even after reordering $(Z_1,\dots,Z_N)$ into    $(Z^{(1)},Z^{(2)},\dots,Z^{(N)})$.}

\end{remark}

\subsection{The very heavy tailed case: proof of \eqref{premconv}}\label{cas-non-L2}
%

Assume $\kappa>\undemi$ and  $d\in \{1,2\}$. Fix  $K\in \N$ and split $f_N$ into two sums, i.e.,  
$f_N=f_N^+ +f_N^-$ where
\begin{align}\label{deffN}
f_N^+(\s)&=\sum_{i\leq N\colon Z_i>KD} Z_i^{\alpha-d}\  \bigg[  \psi\Big(\frac{v_{Y_i}(\s)}{Z_i}\Big)- \psi\Big(\frac{v_{Y_i}(0)}{Z_i}\Big)\bigg] \\
\nonumber f_N^-(\s)&=\sum_{i\leq N\colon Z_i\leq KD} Z_i^{\alpha-d}\   \bigg[  \psi\Big(\frac{v_{Y_i}(\s)}{Z_i}\Big)-  \psi\Big(\frac{v_{Y_i}(0)}{Z_i}\Big)\bigg], \quad \s\in [0,D]^d.
\end{align} 


With the following Lemmas we identify the limit in distribution of $f_N^+$ and $f_N^-$ properly rescaled.
Lemma \ref{resultpart1} and \ref{resultpart2} will be proven in Sections \ref{prlem1} and \ref{prlem2}, respectively.  
We recall the definition of $\mu$ in Section \ref{descrmu}.

\begin{lemma}\label{resultpart1}
If $\kappa>\undemi$, then
$$N^{-\kappa}\ f_N^+\Rightarrow_{N} \mu_d.$$
\end{lemma}

\begin{lemma}\label{resultpart2}
 If $\kappa>\undemi$,  then the
 sequence $(N^{-\frac12} f_N^-)_{N\in\N}$ converges in distribution. Therefore,

$$N^{-\kappa}\  f_N^-\Rightarrow_{N} 0.$$
\end{lemma}

The set $(C_d, ||\cdot ||_\infty)$ is Polish which allows us to apply Slutsky's Lemma. Thus, the convergence \ref{premconv} is a  straightforward 
consequence of Lemmas \ref{resultpart1} and \ref{resultpart2} above.

\subsubsection{Convergence  of $N^{-\kappa} f_N^+$: proof of Lemma \ref{resultpart1}}\label{prlem1}


We set
$$\tau_N:=\max\{i\geq 1\colon\, Z^{(i)}\geq K D\},$$
and we recall the definition of $n_1$ below \eqref{defgamma12}. We spit $f_N^+$ into 
$f_N^1+f_N^2+f_N^3$ where, for $s\in [0,D]$,
\begin{align}\label{deffN}
\nonumber f_N^1(\s)&=\sum_{i=1}^{n_1} \big[Z^{(i)}\big]^{\alpha-d}\  \bigg[ \psi\Big(\frac{v_{Y_{i}}(\s)}{Z^{(i)}}\Big)-  \psi\Big(\frac{v_{Y_i}(0)}{Z^{(i)}}\Big)\bigg] \\
f_N^2(\s)&=\sum_{i=n_1+1}^{\sqrt{N}} \big[Z^{(i)}\big]^{\alpha-d}\  \bigg[ \psi\Big(\frac{v_{Y_{i}}(\s)}{Z^{(i)}}\Big)-  \psi\Big(\frac{v_{Y_i}(0)}{Z^{(i)}}\Big)\bigg] \\
\nonumber f_N^3(\s)&=\sum_{i=\sqrt{N}+1}^{\tau_N} \big[Z^{(i)}\big]^{\alpha-d}\   \bigg[  \psi\Big(\frac{v_{Y_i}(\s)}{Z^{(i)}}\Big)-  \psi\Big(\frac{v_{Y_i}(0)}{Z^{(i)}}\Big)\bigg].
\end{align}

Therefore, we may, without changing the law of the triplet of processes $(f_N^1, f_N^2, f_N^3)$, substitute to $(Z^{(1)},\dots,Z^{(N)})$
the random variables defined in the r.h.s. of \eqref{statordrepalt} that are (for simplicity) denoted by  $(Z_N^{(1)},\dots,Z_N^{(N)})$. 
Thus, we set
\begin{align}\label{defKi}
K_{i}(\s):&=\big[Z^{(i)}_N\big]^{n}\, \bigg[ \psi\bigg(\frac{v_{Y_i}(\s)}{Z^{(i)}_N}\bigg)-\psi\bigg(\frac{v_{Y_i}(0)}{Z_N^{(i)}}\bigg)\bigg]
\end{align}
in such a way that for $\s\in [0,D]^d$ and $j\in \{1,2,3\}$,
\begin{align}\label{altrepfn}
N^{-\kappa}\, f_N^j(\s)&
=\, \Big(\frac{T_{N+1}}{N}\Big)^\kappa\, \sum_{i\in A_N^j}\frac{K_{i}(\s)}{T_i^\kappa},
\end{align}
with $A_N^1=\{1,\dots,n_1\}$,  $A_N^2:=\{n_1+1,\dots,\sqrt{N}\}$  and
$A_N^3:=\{\sqrt{N}+1,\dots,\tau_N\}$.
By Lemma~\ref{lem:loivys}, we know that when $Y$ is uniform on $[0,D]^d$,  the law of $v_Y(s)$ does not depend on $s\in [0,D]^d$. Moreover, for $i\in \N$ the random variables $Y_i$ and $Z_N^{(i)}$ are independent by assumption. Therefore,
the random variables
$(K_i(s))_{s\in[0,D]^d}$ are centered.

\begin{remark}\label{llnea}
{\rm
The law of large number applied to $(\xi_i)_{i\geq 1}$ ensures us that $P$-almost surely, $\lim_{N\to \infty} \frac{T_{N+1}}{T_{\sqrt{N}}}=\infty$. Therefore,  the probability that $\sqrt{N}\leq \tau_N$ converges to $1$ as $N\to \infty$. Since we are looking for a convergence in distribution, it is therefore sufficient to work under the event $\sqrt{N}\leq \tau_N$.
The law of large number also gives that, $P$-almost surely, $\lim_{N\to \infty} \frac{T_{N+1}}{N}=1$ which  implies that it suffices to consider 
the convergence in law of the random continuous processes $\hat f_N^j$ where 
\begin{align}\label{defhat}
\hat f_N^j(\s) &:=\sum_{i\in A_N^j} \frac{K_{i}(\s)}{T_i^\kappa}, \quad \text{for}\  j\in \{1,2,3\},\ \s\in [0,D]^d.
\end{align}}
\end{remark}

\begin{remark}\label{controlfy}
{\rm
Combining Lemma \ref{lem:hng} and the fact that $v_{\y}(\s)\leq dD/2$
for $\y,\s\in [0,D]^d$ we obtain that there 
exists a constant $C>0$ (depending on $\psi$ only) such that 
\begin{equation}\label{control diffpsi}
\Big | \psi\Big(\frac{v_{Y_i}(\s)}{Z_N^{(i)}}\Big)-  \psi\Big(\frac{v_{Y_i}(\x)}{Z_N^{(i)}}\Big)\Big | \leq C \big[Z_N^{(i)}\big]^{-n} |v_{Y_i}(\s)-v_{Y_i}(\x)|, \quad \s,\x\in [0,D]^d.
\end{equation}}

\end{remark}

By using  the inequality \ref{control diffpsi}, we can assert that  for every $i\leq \tau_N$, 
\begin{equation}\label{inegKi}
\valabs{K_{i}(\s)}\leq C |v_{Y_i}(\s)-v_{Y_i}(0)|
\end{equation}
which implies (since $v_{\y}(\s)\leq dD/2$ for $(\y,\s)\in [0,D]^d$) that there exists a $c>0$ such that 
$K_i(\s)\leq c$ for all $i\in \{1,\dots,\tau_N\}$ and $\s\in
[0,D]^d$. Remember that the random variable
$1/T_i^{\kappa}$ is in  $L^2$ 
if and only if  $i>n_1=\lfloor 2 \kappa \rfloor$. This guarantees us that, for $\s\in [0,D]^d$, the random variables $f_N^2(\s)$ and $f_N^3(\s)$
are in $L^2$. This is not the case for $f_N^1(\s)$.

 With the help of Slutsky's Lemma, Lemma \ref{resultpart1} is a straightforward  consequence of both convergences in 
Proposition \ref{convergence-en-loi-processus}  below.

%
%

%
%
%
%
%
  
\begin{proposition}\label{convergence-en-loi-processus}
For $d\in \{1,2\}$ and $\kappa>\undemi$,   
\begin{align}\label{convhatf}
 \hat f_N^1+\hat f_N^2&\Rightarrow_{N} \mu_d\\ 
\nonumber \hat f_N^3&\Rightarrow_N 0.
\end{align}
\end{proposition}
We will first prove Proposition \ref{convergence-en-loi-processus} subject to Lemmas \ref{lem:tkol1} and \ref{lem:conv23} below. Subsequently, we will prove 
those two Lemma in Sections \ref{prooflem1} and \ref{prooflem2} respectively.
To be more specific, with lemma \ref{lem:tkol1}, we apply Kolomogorov criterion (recall Proposition \ref{kolmog}) to check that $(\hat f_N^2)_{N\geq 1}$ and $(\hat f_N^3)_{N\geq 1}$ are tight sequences of random functions.  With Lemma \ref{lem:conv23}, in turn, we prove some convergence in $L^2$ that we will use afterwards to make sure that $(\hat f_N^2)_{N\geq 1}$ and $(\hat f_N^3)_{N\geq 1}$ also converge in finite dimensional distributions.
 \begin{lemma}\label{lem:tkol1}
\item For $d\in \{1,2\}$ and $\kappa>1/2$ there exists a $c>0$ such that for $j\in \{2,3\}$ and for $(\x,\s)\in [0,D]^d$
\begin{equation}\label{tensionkol1}
\underset{N\geq 1}{\sup}  \ \E\Big[\Big| \hat f_N^j(\x)-\hat f_N^j(\s)\Big |^4\Big]\leq c\, |x-s|_1^{4}.
\end{equation}

\end{lemma}

\begin{lemma} \label{lem:conv23}
For $d\in \{1,2\}$ and for $s\in [0,D]^d$
\begin{equation}\label{conv2prem}
\lim_{N\to \infty}\hat f_N^2(\s)\underset{L^2}{=}\gamma_{2,\infty}(\s),
\end{equation}
\item for all  $\s\in [0,D]^d$
\begin{equation}\label{conv3prem}
\lim_{N\to \infty} \hat f_N^3(\s)\underset{L^2}{=}0.
\end{equation}
\end{lemma}

\begin{proof}[Proof of proposition \ref{convergence-en-loi-processus} subject to Lemmas \ref{lem:tkol1} and \ref{lem:conv23}]

\noindent To begin with, we recall \eqref{defgamma12}. Then, we apply \eqref{eq:hng1} in \eqref{defKi} and we obtain  
that  the sequence of random functions $(\hat f_N^1)_{N\geq 1}$
converges $P$- almost surely towards $\gamma_1$ for the $||\cdot||_\infty$ norm. As a consequence $(\hat f_N^1)_{N\geq 1}$
is tight.

We apply Proposition \ref{kolmog}, that is we use Lemma \ref{lem:tkol1} for $j=3$ and \eqref{conv3prem} at $\s=0$ to conclude that $(f_N^3)_{N\geq 1}$ is a tight sequence of random continuous processes. To complete the proof of 
$f_N^3\Rightarrow_N 0$, 
it remains to show that, for every $\s\in [0,D]^d$ (recall Remark \ref{remconvprobponct}) the random sequence  $(f_N^3( \s))_{N\in \N}$ converges in probability towards $0$, but this is a consequence of \eqref{conv3prem}.

With the help of (\ref{tensionkol1}) for $j=2$ and with \eqref{conv2prem} at $\s=0$, we apply Proposition \ref{kolmog} again and we conclude that  $(\hat f_N^2)_{N\geq 1}$ is a tight sequence of random continuous processes.  At the beginning of the proof, we have noticed that 
$(\hat f_N^1)_{N\geq 1}$ is also tight. Thus $(f_N^1+f_N^2)_{N\geq 1}$ is also a tight sequence of continuous process. This can be seen for instance as a consequence of the fact that the continuity modulus of the sum of two functions is bounded above by the sum of the continuity modulus of those functions.   Thus the proof of $\hat f_N^1+\hat f_N^2\Rightarrow_N \mu_d$ will be complete once we prove the convergence in finite dimensional distributions of $f_N^1+f_N^2$ towards $\gamma_1+\gamma_{2,\infty}$.
To that aim, by using Remark \ref{remconvprobponct}, we pick  $ \s\in [0,D]^d$ and  we must show that 
\begin{equation}
\lim_{N\to \infty}\ f_{N}^1(\s)+f_N^2(\s)\underset{\text{Prob}}{=} 
\gamma_1(\s)+\gamma_{2,\infty}(\s).
\end{equation}
The fact that $f_N^1$ converges $P$-almost surely towards $\gamma_1$ for the $||\cdot||_\infty$ norm implies that $\hat f_N^1(\s)$ converges almost surely towards $\gamma_1(\s)$.
Moreover, the fact that $f_N^2(\s)$ converges in probability towards  $\gamma_{2,\infty}(\s)$
is a consequence of \eqref{conv2prem} and this completes the proof. 

%
%

\end{proof}

%
\subsubsection{Proof of Lemma \ref{lem:tkol1}}\label{prooflem1}

Pick $j\in \{2,3\}$ and let us prove \eqref{tensionkol1}. 
We set $K_i(\s,\x)=K_i(\x)-K_i(\s)$ and we denote by 
$\cH_N:=\sigma(T_1,\dots,T_{N+1})$ the $\sigma$-algebra generated by the  random variables $T_1,\dots,T_{N+1}$.  
We condition on  $\cH_N$ and we obtain
\begin{align}\label{cauchy}
 E\Big[\big|\hat f_N^j(\x)-\hat f_N^j(\s)\big|^4\Big] &=
E\bigg[ E\bigg[\bigg(\sum_{i\in A_N^j} \frac{K_i(\s,\x)}{T_i^\kappa} \bigg)^4\Big | \cH_N\bigg]\bigg].
\end{align}
Recall that $n_1=\lfloor 4\kappa\rfloor$ such that $T_i^{-\kappa}\in L^4$ for every $i>n_1$. Conditionally on $\cH_N$, the random variables $(K_i(\s,\x))_{i\in \N}$ are independent and centered because the random variables $(Y_i)_{i \in \N}$ are i.i.d. 
and independent of $\cH_N$. Thus,
after expanding the product in the r.h.s. in \eqref{cauchy}, we
only get the following terms
\begin{align}\label{koldim4}
 \nonumber \esp{ \etp{\sum_{i\in A_N^j} \frac{K_i(\s,\x)}{T_i^\kappa} }^4 \mid
    \Hrond_N} &= \sum_{i\in A_N^j} \esp{K_i(\s,\x)^4} T_i^{-4\kappa}
  \\\
  &+ \sum_{k\neq l \in A_N^j} \esp{K_k(\s,\x)^2 K_l(\s,\x)^2}
  T_k^{-2\kappa}\  T_l^{-2\kappa}.
\end{align}
By using the inequality \ref{control diffpsi}  combined with \eqref{defKi} and since  $v_{\y}(\x)\leq D$ for $(\y,\x)\in [0,D]^d$, we obtain that there exists a
$c_3>0$ such that
\begin{align}\label{boundkk}
\big | K_i(\x,\s)\big |\leq c_3 |v_{Y_i}(\x)-v_{Y_i}(\s)|\leq c_3\,  |\x-\s|_1
\end{align}
where we have also used that $|\text{dist}(\s,Y_i+D\Z^d)-\text{dist}(\x,Y_i+D\Z^d)|\leq |\s-\x|_1$.
\medskip

Combining \eqref{boundkk} with  \eqref{koldim4} we obtain

  \begin{equation}
   \sum_{i\in A_N^j} \esp{K_i(\s,\x)^4} T_i^{-4\kappa}   \le C' |\x-\s|_1^{4} \sum_{i\in A_N^j} T_i^{-4\kappa}.
  \end{equation}
 For the second term in \eqref{koldim4}, we apply Cauchy-Schwarz inequality to write
$$  \esp{T_k^{-2\kappa}\,  T_l^{-2\kappa} } \leq \esp{T_k^{-4 \kappa}}^{1/2}\, \esp{T_l^{-4\kappa}}^{1/2}  $$
so that after taking the expectation in both sides of  \eqref{koldim4} we can rewrite \eqref{cauchy} as
\begin{align}
 E\Big[\big|\hat f_N^j(\x)-\hat f_N^j(\s)\big|^4\Big] &    \le C' |\x-\s|_1^{4} \Bigg[\sum_{i\in A_N^j} \esp{T_i^{-4\kappa}}+  \Big(\sum_{i\in A_N^j} \esp{T_i^{-4\kappa}}^{1/2}\Big)^2\Bigg].
\end{align}

Since $2\kappa>1$ and since for $i \in A^j_N$ we have $i> n_1\ge 4\kappa$, we deduce from
\eqref{equivgamma} that there exists $c_1>0$ and $c_2>0$ such that 
\begin{align} 
 \sum_{i\in A_N^j}  \esp{T_i^{-4\kappa}}^{1/2} \le  \sum_{i\ge n_1}
  \esp{  T_i^{-4\kappa}}^{1/2} < c_1 \sum_{i=1}^\infty \frac{1}{i^{2\kappa}}<+\infty\\
   \sum_{i\in A_N^j}  \esp{T_i^{-4\kappa}} \le  \sum_{i\ge n_1}
  \esp{  T_i^{-4\kappa}} < c_2 \sum_{i=1}^\infty \frac{1}{i^{4\kappa}}<+\infty.
\end{align}


and therefore that for a constant $C''>0$,
\begin{equation}\label{rectov}
\underset{N\geq 1}{\sup}  \ \E\Big[\Big| \hat f_N^j(\x)-\hat f_N^j(\s)\Big |^4\Big]\leq C''\, |\x-\s|_1^{4}
\end{equation}

\subsubsection{Proof of Lemma \ref{lem:conv23}}\label{prooflem2}
Now, we keep going with \eqref{conv2prem}. We are going to prove that $\hat f_N^{2}(\s)-\gamma_{2,\sqrt{N}}(\s)$ converges to $0$ in $L^2$, which combined 
with Proposition \ref{convgamma2N} will be sufficient to complete the proof of \eqref{conv2prem}.
We recall \eqref{defKi} and we use  \eqref{eq:hng1} at $x=v_{Y_i}(\s)/Z_N^{(i)}$ and at $x=v_{Y_i}(0)/Z_N^{(i)}$ to obtain 
\begin{equation}\label{difkg}
K_i(\s)-G_i(\s)=\gep\bigg(\frac{v_{Y_i}(\s)}{Z_N^{(i)}}\bigg)\, v_{Y_i}(\s)^n-\gep\bigg(\frac{v_{Y_i}(0)}{Z_N^{(i)}}\bigg)\, v_{Y_i}(0)^n
\end{equation}
Using \eqref{difkg}, combined with the fact that, conditionally on $\cH_N$, the random variables $(K_i(\s)-G_i(\s))_{i\geq n_1+1}$ 
are independent and centered, we obtain that there exists a $C>0$ (depending on $D$ and $n$ only) such that 
\begin{align}\label{norml2kg}
\nonumber ||\hat f_N^{2}(\s)-\gamma_{2,\sqrt{N}}(\s)||_2^2&=\sum_{i=n_1+1}^{\sqrt{N}}  E\Bigg[\frac{1}{T_i^{2\kappa}} \Bigg(\gep\bigg(\frac{v_{Y_i}(\s)}{Z^{(i)}_N}\bigg)\ v_{Y_i}(\s)^n-\gep \bigg(\frac{v_{Y_i}(0)}{Z^{(i)}_N}\bigg) \ v_{Y_i}(0)^n\Bigg)^2 \Bigg]\\
&\leq C \sum_{i=n_1+1}^{\infty}  E\Bigg[\frac{1}{T_i^{2\kappa}}\max_{u\in [0,dD/2]}\bigg |\gep\bigg(\frac{u}{Z^{(\sqrt{N})}_N}\bigg)\bigg |^2\Bigg].
\end{align}
At this stage, we recall Remark \ref{llnea} which states that
$P$-almost surely 
$\lim_{N\to \infty} Z_N^{(\sqrt{N})}=\infty$. Since $\lim_{u\to 0} \epsilon(u) =0$ and
$\kappa > \undemi$, we can use dominated convergence to prove that
$\lim_{N\to \infty} ||\hat f_N^{2}(\s)-\gamma_{2,\sqrt{N}}(\s)||_2=0$.
Indeed, since we restricted ourselves to the set $\ens{\sqrt{N} \le
    \tau_N}$ on which $Z^{(\sqrt{N})}_N \ge K D$, we have
  \begin{equation}
    \sup_{N\in \N} \sup_{u\in[0, dD/2]} \valabs{\epsilon\etp{\frac{u}{Z_N^{(\sqrt{N})}}}}
    \le \sup_{v\in [0,d/2K]} \valabs{\epsilon(v)} < +\infty\,.
  \end{equation}
It remains to use Proposition \ref{convgamma2N} which guarantees us that $(\gamma_{2,\sqrt{N}})_{N\geq 1}$ converges towards 
$\gamma_{2,\infty}$ in $L^2$ to complete the proof of \eqref{conv2prem}.

Let us now prove \eqref{conv3prem}. It is enough to show that $\lim_{N\to \infty}  ||\hat f_N^{3}(\s)||_2=0$. To that aim, we write
\begin{align}\label{cauchy}
\nonumber E\Big[\hat f_N^{3}(\s)^2\Big] &=
E\bigg[ \ E\bigg[\bigg(\sum_{i=\sqrt{N}+1}^{\tau_N} \frac{K_i(\s)}{T_i^{\kappa}} \bigg)^2\Big | \cH_N\bigg]\bigg]\\
 &=
E\bigg[ \sum_{k=\sqrt{N}+1}^N \  \ind_{\{\tau_N=k\}}\  E\bigg[\bigg(\sum_{i=\sqrt{N}+1}^{k} \frac{K_i(\s)}{T_i^{\kappa}} \bigg)^2\Big | \cH_N\bigg]\bigg].
\end{align}
In order to mimic the computation \eqref{koldim4}, we observe that, conditionnally on $\cH_N$, the random variables  $(K_i(\s))_{i\in \N}$ are independent and centered since  $(Y_i)_{i\in\N}$ are independent of $\cH_N$. Thus, for $k\in \{\sqrt{N}+1,\dots,\tau_N\}$, we obtain 
\begin{align}\label{sommecondhn}
\nonumber E\bigg[\bigg(\sum_{i=\sqrt{N}+1}^{k} \frac{K_i(\s)}{T_i^{\kappa}} \bigg)^2\Big | \cH_N\bigg]
&=\sum_{i=\sqrt{N}+1}^{k} E\big[ K_i(\s)^2\big] \frac{1}{T_i^{2\kappa}}\\
&\leq c \sum_{i=\sqrt{N}+1}^{k}  \frac{1}{T_i^{2\kappa}}
\end{align}
where, to obtain the second line, we have used that there exists a $c>0$ such that  $|K_i(\s)|\leq c$ for $\s\in [0,D]^d$ and $i\leq \tau_N$ (this was proven below \eqref{inegKi}).
It remains to use \eqref{sommecondhn} in \eqref{cauchy} to obtain  
\begin{equation}\label{convalt}
\nonumber E\Big[\hat f_N^{3}(\s)^2\Big]  \leq c
\sum_{i=\sqrt{N}+1}^{\infty} E\Bigg[ \frac{1}{T_i^{2\kappa}}\Bigg]\,,
\end{equation}

which, combined with \eqref{equivgamma} and with the inequality
$2\kappa>1$ allows us to conclude that
\begin{equation}
  \lim_{N\to \infty} E\Big[\hat f_N^{3}(\s)^2\Big] =0\,.
\end{equation}

\smallskip

\subsubsection{Convergence of $N^{-\kappa}\, f_N^{-}$: proof of Lemma \ref{resultpart2}}
\label{prlem2}

We will prove that $(\frac{1}{\sqrt{N}} f_N^{-})_{N\in \N}$ is a tight sequence of random functions. Thus, with the help of 
Slutsky's lemma, it will be sufficient to prove Lemma \ref{resultpart2} because $\kappa>1/2$. In order to obtain this tightness in dimension $d=1$
and $d=2$ we will use Proposition \ref{kolmog}. 

 We recall \eqref{elemprotrans} and we write
\begin{align}
\unsur{\sqrt{N}} f_N^-(\x)=\unsur{\sqrt{N}} \sum_{i=1}^N \un{Z_i\leq KD}\   \Big[X_i(\s)-X_i(0)\Big].
\end{align}
We observe that $f_N^{-}(0)=0$ for every $N\in \N$ so that  $(\frac{1}{\sqrt{N}} f_N^{-}(0))_{N\in \N}$ is obviously tight.
It remains to consider for $d=1$ and $(s,x)\in [0,D]^2$,
\begin{align}\label{compsecmome}
\nonumber  \unsur{N} \esp{\valabs{f_N^-(x) -f_N^-(s)}^2} &\leq  \frac{1}{N}\  \esp{ \Big(\sum_{i=1}^N  \un{Z_i\leq KD}\   \Big[X_i(s)-X_i(x)\Big]\Big)^2}\\
  &=  \frac{1}{N}\ \sum_{i=1}^N \esp{   \un{Z_i\leq KD}\   \Big[X_i(s)-X_i(x)\Big]^2},
\end{align}
where the equality in \eqref{compsecmome} is true because, on the one hand,
the sequence of random variables $\big(\un{Z_i\leq KD}\   \big[X_i(s)-X_i(x)\big]\big)_{i\in \N}$ is independent and bounded, and on the other hand,
Lemma \ref{lem:loivys} guarantees us that the law of $v_{Y_i}(t)$ does not depend on $t\in [0,D]$. Therefore, the sequence of random variables is also identically distributed and centered.  From \eqref{compsecmome} combined with Lemma  \eqref{eq:hng2} in Lemma \ref{lem:hng} we deduce that there exists $C>0$ and $C'>0$ such that
\begin{align}\label{boundsecmomfnsmall}
\nonumber  \unsur{N} \esp{\valabs{f_N^-(x) -f_N^-(s)}^2}
&\le C\  \esp{\un{Z_1 \le
      KD} Z_1^{2(\alpha-n-d))}}
      |x-s|^{2}\\
      &= C'\  |x-s|^{2},
\end{align}
and note that we have used again that $|\text{dist}(s,Y_i+D\Z)-\text{dist}(x,Y_i+D\Z)|\leq |s-x|_1$. This 
completes the proof of the tightness of   $(\frac{1}{\sqrt{N}} f_N^{-})_{N\in \N}$ for $d=1$.

It remains to prove the counterpart of \eqref{boundsecmomfnsmall} when $d=2$. For the sake of conciseness we set for $i\in \N$
and $\s \in [0,D]^2$,
$$\widetilde X_i(\s)=\un{Z_i\leq KD} \ X_i(\s).$$
For $N\in \N$ and $(\s,\x)\in [0,D]^2$,
we write that there exists $C>0$ such that
 \begin{align}\label{proofdim2bound}
\nonumber  \frac{1}{N^2}\, E\Big[ &\big| f_N(\x)-f_N(\s)\big |^4\Big]=\frac{1}{N^2} E\Big[ \Big(\sum_{i=1}^N \big[\widetilde{X}_i(\x)-\widetilde{X}_i(\s)\big]\Big)^4\Big]\\
\nonumber  &=\frac{1}{N^2}\sum_{i=1}^N E(|\widetilde{X}_i(\x)-\widetilde{X}_i(\s)|^4)+\frac{2}{N^2}\, \sum_{1\leq i<j\leq N} E(|\widetilde{X}_i(\x)-\widetilde{X}_i(\s)|^2 |\widetilde{X}_j(\x)-\widetilde{X}_j(\s)|^2 )\\
\nonumber  &=\frac{1}{N}\, E(|\widetilde{X}_1(\x)-\widetilde{X}_1(\s)|^4)+\frac{3 N (N-1)}{N^2} E(|\widetilde{X}_1(\x)-\widetilde{X}_1(\s)|^2 |\widetilde{X}_2(\x)-\widetilde{X}_2(\s)|^2 )\\
 & \leq  C |\x-\s|_1^4
 \end{align}
where we have use that, for the same reasons as in dimension $1$, the random variables $(\widetilde X_i(\x)-\widetilde X_i(\s))_{i\in \N}$ are i.i.d., bounded and centered. This completes the proof in dimension $d=2$.

\subsection{ Proof of  Theorem \ref{fluctconver}, case
  $\kappa<1/2$.}\label{cas-L2}

The proof of \eqref{secconv} will pretty much look alike that of  Lemma \eqref{resultpart2} except that  we work here with the sequence of 
independent random variables $ \big[X_i(\s)-X_i(\x)\big]\big)_{i\in \N}$ instead of $\big(\un{Z_i\leq KD}\   \big[X_i(\s)-X_i(\x)\big]\big)_{i\in \N}$. However, the fact that $\kappa<1/2$ guarantees us that $X_1(\s)-X_1(\x)$ is in $L^2$.
\subsubsection{Finite dimensional convergence}
We recall Remark \ref{explibetac} and more specifically equations \eqref{boundzx} and \eqref{L2forz} which guarantees us that (since $\kappa<1/2$) the random variables 
$X_1(s)-X_1(0)$ are in $L^2$. As a consequence, we can claim  that for every  $\bar s\in ([0,D]^d)^k$ the sequence of random vectors $(X_i(\bar s)-X_i(0_k))_{i\geq 1}$ is i.i.d., centered and in $L^2$. 
Therefore, a straightforward application of the 
multi-dimensional central limit Theorem ensures us that any finite dimensional marginal of $N^{-1/2}\, f_N$ converges to $Y$ (recall
the definition of $Y$ below \eqref{secconv})
  as $N\to \infty$.
\subsubsection{Convergence of processes, case $d=1$ }
 Finite dimensional convergence established, only the tightness of the
 sequence of processes need to be proved, so  
\eqref{secconv} will be proven once we show that there exist a $C>0$ such that for every $(x,s)\in [0,D]^2$, 
\begin{equation}\label{tensionkolcastcl}
\underset{N\geq 1}{\sup} \   \frac{1}{N}\,  \E\Big[\Big| f_N(x)- f_N(s)\Big |^2\Big]\leq C\, |x-s|^{2}.
\end{equation}


We repeat the proof of \eqref{boundsecmomfnsmall} with the slight difference that we must remove the terms $\un{Z_1 \leq KD}$ but this does not arm the proof since (as explained above) $\kappa<1/2$  yields  that $X_1(s)-X_1(0)$ are in $L^2$. 
As a consequence, for $(x,s)\in [0,D]^2$,
 \begin{align}
\nonumber  \frac{1}{N}\, E\Big[ \big| f_N(x)-f_N(s)\big |^2\Big]&=\frac1N  \sum_{i=1}^N E(|X_i(x)-X_i(s)|^2)\\
\nonumber  &=E(|X_1(x)-X_1(s)|^2)
 \end{align}
 With the help of Remark \ref{controlfy}, we write
 \begin{align}\label{fluctuations}
  \nonumber  \esp{(X_1(x)-X_1(s))^2}&=\esp{Z_1^{2(\alpha-d)}
      \etp{\psi\Big(\frac{v_{Y_1}(s)}{Z_1}\Big) -
        \psi\Big(\frac{v_{Y_1}(x)}{Z_1}\Big)}^2}\\
\nonumber     &\le C \esp{Z_1^{2(\alpha-n-d)} |v_{Y_1}(s)
        -v_{Y_1}(x)|}^{2} \\
\nonumber     &\le C \esp{Z_1^{2(\alpha-n-d)}}
      |x-s|^{2}\\
      &= C' |x-s|^{2}\,.
 \end{align}
 This completes the proof in the case $d=1$.

\subsection{Convergence of the processes, case $d=2$, $\kappa < \frac14$}\label{proofcasedis2}
We shall establish tightness of the sequence of processes. 
 We show that there exists a $C>0$ such that for every $(\x,\s)\in ([0,D]^2)^2$
\begin{equation}\label{tensionkolcastcl}
\underset{N\geq 1}{\sup} \   \frac{1}{N^2}\,  \E\Big[\Big| f_N(\x)- f_N(\s)\Big |^4\Big]\leq C\, |\x-\s|_1^{4}.
\end{equation}


Let us prove \eqref{tensionkolcastcl} by mimicking \eqref{proofdim2bound}, i.e.,
 \begin{align}\label{fourthmoment}
\nonumber  \frac{1}{N^2}\, E\Big[ \big| f_N(\x)&-f_N(\s)\big |^4\Big]=\frac{1}{N^2} E\Big[ \Big(\sum_{i=1}^N X_i(\x)-X_i(\s)\Big)^4\Big]\\
 &=\frac{1}{N^2}\sum_{i=1}^N E(|X_i(\x)-X_i(\s)|^4)+\frac{2}{N^2}\, \sum_{1\leq i<j\leq N} E(|X_i(\x)-X_i(\s)|^2 |X_j(\x)-X_j(\s)|^2 )\\
\nonumber  &=\frac{1}{N}\, E(|X_1(\x)-X_1(\s)|^4)+\frac{3 N (N-1)}{N^2} E(|X_1(\x)-X_1(\s)|^2 )^2.
 \end{align}
It remains to apply Remark \ref{controlfy} to obtain  that there exists $C>0$ such that
\begin{align}\label{upperboundmom}
E(|X_1(\x)-X_1(\s)|^4)&\leq C\  E\big[Z_1^{4(\alpha-d-n)}\big] |\x-\s|_1^4, \\
E(|X_1(\x)-X_1(\s)|^2)&\leq C E\big[Z_1^{2(\alpha-d-n)}\big] |\x-\s|_1^2.
\end{align}
The fact that $\kappa<1/4$ guarantees us that $E\big[Z_1^{2(\alpha-d-n)}\big]$ and $E\big[Z_1^{4(\alpha-d-n)}\big]$ are finite.
Then, it suffices to combine \eqref{fourthmoment} with \eqref{upperboundmom}
to achieve the proof of \eqref{tensionkolcastcl}.

\subsection{Critical case $d\in\ens{1,2},\kappa=\undemi$}\label{subsec:criticalrand}
We will first prove the tension of  \( L_N = (N \log N)^{-1/2} f_N \), which we split into $L_N:=L_N^1+L_N^2$ associated with the decomposition 
\[
\begin{aligned}
& f_N^1(\x) = \sum_{k=1}^N \left( X_k(\x) - X_k(0) \right) \un{ Z_k^{\alpha-d-n} \leq \sqrt{k} }.   \\
& f_N^2(\x) = \sum_{k=1}^N \left( X_k(\x) - X_k(0) \right) \un{ Z_k^{\alpha-d-n} > \sqrt{k} }.
\end{aligned}
\]
We will use Proposition \ref{tightness1}, which involves the modulus of continuity. We start with $(L_N^2)_{N\in \N}$
and we observe that $Z_1^{\alpha-d-n}$ is Pareto distributed with parameter $2$. Thus,
\[
\mathbb{E}\left[ Z^{\alpha-d-n} \un{ Z^{\alpha-d-n} > \sqrt{k} } \right]=  \frac{2}{\sqrt{k}}.
\]
and an upper bound obtained with  remark \ref{controlfy}  gives us that there exists a $C_1>0$ such that 
\[
\left| f_N^2(x) - f_N^2(y) \right| \leq C_1 \sum_{k=1}^N Z_k^{\alpha-d-n} \un{ Z_k^{\alpha-d-n} > \sqrt{k} }, \quad (x,y)\in [0,D]^d\times [0,D]^d.
\]
Therefore, there exists a $C_2>0$ such that
\[
\begin{aligned}
\mathbb{P}\left( w_{L_N^2}(\delta) \geq \epsilon \right) & \leq \mathbb{P}\left( C_1 \sum_{k=1}^N Z_k^{\alpha-d-n} \un{ Z_k^{\alpha-d-n} > \sqrt{k} } \geq  \gep (N \log N)^{\frac12} \right) \\
& \leq \frac{C_1}{\gep} (N \log N)^{-\frac12}  \mathbb{E}\left[ \sum_{k=1}^N Z_k^{\alpha-d-n} \mathbf{1}_{\left\{ Z_k^{\alpha-d-n} > k \right\}} \right] \\
& \leq \frac{C_1}{\epsilon} (N \log N)^{-\frac12} \sum_{k=1}^N \frac{2}{\sqrt{k} } \\
& \leq \frac{C_2}{\gep} \sqrt{\frac{1}{\log N}}.
\end{aligned}
\]

Consequently,
\begin{equation}\label{upperboundmodcont}
\lim_{\delta \to 0} \limsup_{N \to +\infty} \mathbb{P}\left( w_{L_N^2}(\delta) \geq \epsilon \right) = 0.
\end{equation}
For the tightness of \( (L_N^1)_{N\in \N} \), we will use the Kolmogorov criterion exposed in Proposition \ref{kolmog}. More precisely, 
we will show that there exists a \( C>0 \) such that
\begin{equation}\label{contrsup}
\sup_{N \geq 1} \mathbb{E}\left[ \left( L_N^1(\x) - L_N^1(\y) \right)^4 \right] \leq C |\x - \y|_1^4 \quad (\x, \y \in [0, D]^d),
\end{equation}
and this will give the desired tightness for $d\in \{1,2,3\}$. For larger dimensions, one has to compute higher moments in \eqref{contrsup} but we will not 
display such computations below.

%
%

Notice that
\[
f_N^1(x) - f_N^1(y) = \sum_{k=1}^N R_k:= \sum_{k=1}^N \left( X_k(x) - X_k(y) \right) \un{ Z_k^{\alpha-d-n} \leq \sqrt{k} }
\]
is a sum of \( N \) independent, centered random variables (the variables $R_k$ are centered because \( X_k(x) \) and \( X_k(y) \) have the same distribution). 
Consequently,
\begin{align}\label{fnlesvar}
\mathbb{E}\left[ \left( f_N^1(x) - f_N^1(y) \right)^4 \right] & = \sum_{k=1}^N \mathbb{E}\left[ R_k^4\right] +2\sum_{1\leq i<j\leq N} E\left[R_i^2\right]\, E\left[R_j^2\right].
\end{align}
With the help of \eqref{control diffpsi} we can assert that there exists a $C>0$ such that 
$$|R_k|\leq  C |x-y|  Z_k^{\alpha-d-n}  \un{ Z_k^{\alpha-d-n} \leq \sqrt{k} },$$
and therefore, after recalling that $Z_k^{\alpha-d-n}$ follows a Pareto distribution of parameter $2$ we write that there exists a $C>0$ such that 
for $k\in \N$,
\begin{align}\label{momtwoandfour}
E\left[R_k^4\right]\leq C\, |x-y|^4\, k \quad \text{and} \quad  E\left[R_k^2\right]\leq C\, |x-y|^2\, \log k .
\end{align}
Putting together \eqref{fnlesvar} and \eqref{momtwoandfour} we obtain that there exists  $C_1>0$ and $C_2>0$  such that 
\begin{align}\label{finuppbound}
\nonumber \mathbb{E}\left[ \left( L_N^1(x) - L_N^1(y) \right)^4 \right] &\leq \frac{C_1}{N^2\, (\log N)^2}\,  |x-y|^4 \,  \Big( \sum_{k=1}^N k + \sum_{1\leq i<j\leq N} \log i \,   \log j \Big),\\
&\leq C_2\,  |x-y|^4.
\end{align}
This, combined with \eqref{upperboundmodcont}  completes the proof of the tightness of $(L_N)_{N\in \N}$.

\medskip

It remains to prove its convergence in finite dimensional distribution.
To that aim, we first apply Lemma \ref{lem:att_par_normal} to the random
variable
\begin{equation}
  \xi_s = X_1(s) -X_1(0) = Z_1^{\alpha -d} \Big[ \psi\Big(\frac{v_s(U_1)}{Z_1}\Big) -\psi\Big(\frac{v_0(U_1)}{Z_1}\Big) \Big]
\end{equation}
with $U_1$ a uniform on $[0,D]^d$ independent of $Z_1$ a Pareto($\beta -1$)
random variable. Since $\kappa=\frac{\alpha -n -1}{\beta -1} =
\frac12$, we have with $A_s = \frac{\psi^{(n)}(0)}{n!}(v_s(U_1)^n -
v_0(U_1)^n)$, according to Lemma \ref{lem:hng},
\begin{align*}
  \prob{\xi_s >x} &= \prob{Z_1^{\alpha -n -1}>
    \frac{x}{A_s}, A_s >0}\, (1+o_{x\to \infty}(1))\\
  &= x^{-2}\,  \esp{A_s^2 \un{A_S >0}}\, (1+o_{x\to \infty}(1)).
\end{align*}
Similarly
\begin{equation}
\prob{-\xi > x}= x^{-2}\, \esp{A_s^2 \un{A_S <0}}\,(1+o_{x\to \infty}(1)).
\end{equation}
Therefore, we obtain that
\begin{equation}
\lim_{N\to \infty}  (N \log N)^{-\undemi} f_N(s)\underset{\text{law}}{=} \Nrond(0,\esp{A_s^2}).
\end{equation}

Given $s_1, \ldots,s_n \in [0,D]$ and real numbers
$\alpha_1,\ldots,\alpha_n$, we can apply the same technique to the
random variable
\begin{equation}
  \xi = \sum_{i=1}^n \alpha_i \xi_{s_i}
\end{equation}
and obtain the convergence, with $A= \sum_{i=1}^n A_{s_i}$,
\begin{equation}
 \lim_{N\to \infty} (N \log N)^{-\undemi} \sum_{i=1}^n f_N(s_i) \underset{\text{law}}{=}\Nrond(0,\esp{A^2})\,.
\end{equation}
We  infer from this the convergence for finite dimensional distributions
of the process $ (N \log N)^{-\undemi} f_N$ to the Gaussian centered
process $(Y(s))_{s\in[0,D]^d}$ of covariance $r(s,t) = \cov(A_s,A_t)$.

\smallskip

\subsection{Tightness of the min fluctuations for $\kappa>1$: proof of Proposition \ref{tightkappalar1}}\label{tightkappalar11}

We use the continuity modulus of $\frac{1}{N^\kappa} f_N^{\text{m}}$ to prove its tightness, that is we use Theorem \ref{tightness1}.  To that aim, we use \eqref{control diffpsi} and for $\x,\y\in [0,D]^d$ we obtain
\begin{align}\label{boundfistepmodcont}
\nonumber \frac{1}{N^\kappa} |f_N^{\text{m}}(\x)-f_N^{\text{m}}(\y)|&\leq \frac{1}{N^\kappa} \sum_{i=1}^N Z_i^{\alpha-d}\  \Big|\psi\Big(\frac{v_{M_i}(\x)}{Z_i}\Big)- \psi\Big(\frac{v_{M_i}(\y)}{Z_i}\Big)\Big|\\
\nonumber &\leq \frac{C}{N^\kappa} \sum_{i=1}^N Z_i^{\alpha-d-n}|v_{M_i}(\x)-v_{M_i}(\y)|\\
\nonumber &\leq \frac{C}{N^\kappa} \sum_{i=1}^N Z_i^{\alpha-d-n}|\x-\y|\\
&\leq C \frac{T_{N+1}^\kappa}{N^{\kappa} }|\x-\y| \sum_{i=1}^\infty \frac{1}{T_i^\kappa},
\end{align}
where we have used the representation of the order statistics of $(Z_1,\dots,Z_N)$ in \eqref{statordrepalt}. 
At this stage, we set $W:=\sum_{i=1}^\infty  (1/T_i)^\kappa$ that is almost surely finite since $\kappa>1$.  For every $\delta>0$ and $\gep>0$, \eqref{boundfistepmodcont} allows us to state that
$$P\Big(w\Big[\frac{1}{N^\kappa} f_N^\text{m}\Big](\delta)\geq \gep\Big)\leq P\Big( C \big(\tfrac{T_{N+1}}{N}\big)^\kappa W  \geq \frac{\gep}{\delta}\Big). $$
The fact that $T_{N+1}/N$ converges $P$-almost surely towards $1$ suffices to assert that \eqref{tension} holds true. This completes the proof of Proposition \ref{tightkappalar1}
since obviously $f_N^\text{m}(0)=0$ for every $N\in \N$.

\section{Fluctuations of the stellar model: Theorem \ref{fluctconverstel}.}\label{Preuvesstel} 

\subsection{Case $\kappa>1/2$}

For $N\in \N$, we denote by $f_N^{\,\text{stel}}=(f_N^{\,\text{stel},1},f_N^{\, \text{stel},2}, f_N^{\,\text{stel},3})$ the $3$ coordinates of the fluctuations for the stellar model after $N$ iterations of the 
transformations.   We recall Section \ref{vectconv} which guarantees us that obtaining \eqref{premconvevect} 
requires to check that for every $j\in \{1,2,3\}$ 
\begin{itemize}
\item the sequence of random functions $(\frac{1}{N^\kappa} f^{\,\text{stel},j}_N)_{N\in \N}$  is tight in $(C_2,||\cdot||_{\infty})$
\item  for every $\s\in [0,D]^2$ the sequence of real random variables 
$(\frac{1}{N^\kappa} f^{\,\text{stel},j}_N(\s))_{N\in \N}$ converges in probability towards  $\gamma^{\text{stel},j}_1(\s)+\gamma^{\text{stel},j}_{2,\infty}(\s)$.
\end{itemize}
In the case $\kappa>1/2$, proving the tightness  amounts to proving Lemma \eqref{lem:tkol1} after multiplying by $\Theta_i^j$
the quantities $K_j(\s)$ in \eqref{defKi}. This does not bring any additional difficulty since 
the random variables $(\Theta_i^j)_{i\in \N}$ are i.i.d., centered, bounded
and independent of $Z$ and $Y$.  Proving the convergence in probability is achieved by adapting the proof of Lemma \ref{lem:conv23} in the stellar framework. For the same reason, the proof
is completely identical.

 \subsection{Case $\kappa<1/2$, finite dimensional convergence.}

It remains to prove the convergence in finite dimensional distribution. To that aim we consider 
$F_N(\bar \x)=(f_N^{\text{stel}}(\x_1),\dots,f_N^{\text{stel}}(\x_d))$. We also consider 
$\Lambda=(\lambda_1,\dots,\lambda_d)\in (\R^3)^d$. Then
\begin{align}\label{tclmultvardim3}
\Big\langle \Lambda,\frac{F_N(\bar \x)}{\sqrt{N}} \Big\rangle&=\frac{1}{\sqrt{N}}\sum_{j=1}^d \langle \lambda_j,f^{\text{stel}}_N(\x_j)\rangle \\
&=\frac{1}{\sqrt{N}} \sum_{i=1}^N \sum_{j=1}^d X_i(\x_j)\, \langle \lambda_j,\Theta_i\rangle .
\end{align}
At this stage, we observe that $(\sum_{j=1}^d X_i(\x_j)\, \langle\lambda_j,\Theta_i \rangle)_{i\in \N}$ is an i.i.d. 
sequence of random variables that are in $L^2$ and centered.  Thus, we can apply the central limit theorem and conclude that
\begin{equation}\label{convenloigauss}
\lim_{N\to \infty}\Big\langle \Lambda,\frac{F_N(\bar \x)}{\sqrt{N}} \Big \rangle \underset{\text{law}}{=} \mathcal{N}\bigg(0,\text{Var} \Big(\sum_{j=1}^d X_1(\x_j)\, \langle\lambda_j,\Theta_1\rangle \Big)\bigg)
\end{equation}
To compute the variance in the r.h.s. in \eqref{convenloigauss}, we recall that $\Theta_1$ is independent of $X_1$ and that  for every $j\in \{1,\dots,d\}$ the random variable $X_1(\x_j)$ is centered.
Therefore,
\begin{align}
\nonumber \text{Var} \Big(\sum_{j=1}^d X_1(\x_j)\, \langle\lambda_j,\Theta_1\rangle \Big)&=\sum_{1\leq j,k\leq d}  E[X_1(\x_j)\, X_1(\x_k)]  \, E(\langle \lambda_j, \Theta_1\rangle \, \langle \lambda_k, \Theta_1\rangle  )   \\
& = \sum_{1\leq j,k\leq d}  E[X_1(\x_j)\, X_1(\x_k)]  \  \lambda_j E[ \Theta_1^{\intercal} \Theta_1] \ \lambda_k^{\intercal}
\end{align}
and this completes the proof of \eqref{terconvevect}.
\subsection{Case $\kappa < \frac14$ : proof of tightness.}
 Concerning the tightness of each coordinate $(\frac{1}{\sqrt{N}} f_N^{\,\text{stel},j})_{N\in \N}$, we simply repeat the proof 
in Section \ref{proofcasedis2} and nothing changes except for the fact that (for every $i\in \N$) the random variable $X_i(\x)-X_i(\s)$ is replaced by $\Theta_i^j \,[X_i(\x)-X_i(\s)]$ 
which is not a problem since (as mentioned in the previous case) $(\Theta_i^j)_{i\in \N}$ is an i.i.d. sequence of centered and bounded random variables, independent of
$X$.
\subsection{Critical case $\kappa=\undemi$.}
For each coordinate $i$, following the proof of the critical case for
the rand model, see section \ref{subsec:criticalrand}, we establish
finite dimensional convergence of $(\frac{1}{\sqrt{N}}
f_N^{\,\text{stel},i})_{N\in \N}$, by considering the random variables
\begin{equation}
  \xi_i(s) = \Theta_1^i (X_1(s) -X_1(0))\,,
\end{equation}
and the result follows along the same lines.

\section*{Acknowledgments}

This work has been supported by the Agence Nationale de la Recherche (Grant No. ANR-24-CE56-3575).

\appendix
\section{Useful Lemmas}
We recall the definition of function $v$ in \eqref{defaltv}.

\begin{lemma}\label{lem:loivys}
\begin{enumerate}
\item In dimension $1$, let $Y\sim\Urond([0,D])$ be distributed as a uniform on
  $[0,D]$. Then, for all $s$, $v_Y(s) \sim \Urond([0,D/2])$ is
  distributed as a uniform on $[0,D/2]$.
 \item In dimension $2$, let $Y\sim\Urond([0,D]^2)$ be distributed as a uniform on
  $[0,D]^2$. Then, for all $\s$, $v_Y(\s)$ is a continuous random variable with density 
  \begin{equation}\label{densityvdimtwo}
 dP_{v_Y(\s)}(u)=\frac{4}{D^2} \Big[ \ind_{[0,D/2]}(u)+ \ind_{[D/2,D]}(u)  (D-u)\Big] du.
  \end{equation}
  \end{enumerate}
\end{lemma}
\begin{proof}
  Recall that for $x,y \in \R$
  \begin{equation}
    v_x(y) =v_y(x)= d(x-y,D\Z)
  \end{equation}
  is the distance form the point $x-y$ to the set $D \Z$. Let $h:\R\to
  \R$ be measurable non negative, and $Y\sim\Urond([0,D])$. We have:
  \begin{align}
    \esp{h(v_Y(s))} &= \unsur{D} \int_0 ^D h(d(y-s,D\Z))\, dy \\
    &= \unsur{D} \int_{-s}^{D-s} h(d(v,D\Z))\, dv \\
    &= \unsur{D} \int_0 ^D h(d(v,D\Z))\, dv,
  \end{align}
  since the function  $v \to d(v, D \Z)$ is $D$-periodic. Moreover,
  for $v \in (0,D)$:
  \begin{equation}
    d(v, D \Z) = v \un{v\le D/2} + (D-v) \un{v> D/2}\,.
  \end{equation}
  Therefore,
  \begin{align}
    \esp{h(v_Y(s))} &= \unsur{D} \etp{\int_0^{D/2} h(v)\, dv +
      \int_{D/2}^D h(D-v)\, dv} \\
    &= \frac{2}{D} \int_0^{D/2}  h(v)\, dv\,.
  \end{align}
\end{proof}

\begin{lemma}\label{lem:hng}
  If $\psi \in \cH^n$ then there exists a constant $C\in(0,\infty)$,
  and a function $\ee:\R^+ \to \R$ such that $\lim_{u\to 0} \ee(u)=0$ and
  \begin{align}
    \psi(x) -\psi(0) &= \frac{x^n}{n!} \psi^{(n)}(0) (1+\ee(x)) &(x
    \ge 0)  \label{eq:hng1}\\
    \valabs{\psi(x) -\psi(y)} &\le C (x\vee y)^{n-1}
    \valabs{x-y} &(x,y\ge 0)\,.\label{eq:hng2}
  \end{align}
\end{lemma}
\begin{proof}
  For $n=1$ there is nothing to prove. Assume $n\ge 2$. Then equation
  \eqref{eq:hng1} is just Taylor's theorem. We then use Taylor's
  formula with an integral form for the remainder to get, for $m=n-1$,
  \begin{equation}
    \psi(x) -\psi(0) = \int_0^x \frac{t^{m-1}}{(m-1)!} \psi^{(m)}
    (x-t)\, dt
  \end{equation}
  Therefore, for $0\le x\le y$, by the Lipshitz property of $\psi^{
    (m)}$ and the fact that $\psi^{(m)}(0)=0$
  \begin{align}
    \valabs{\psi(y) -\psi(x)} &\le \valabs{\int_0^x
      \frac{t^{m-1}}{(m-1)!} (-\psi^{(m)}(x-t)+\psi^{(m)}(y-t)) dt} +\valabs{
      \int_x^y \frac{t^{m-1}}{(m-1)!} \psi^{(m)}(y-t)\,dt}\\
    &\le C \valabs{y-x} \frac{x^m}{m!} + C \int_x^y
    \frac{t^{m-1}}{(m-1)!} (y-t) \, dt \\
    &\le \frac{2 C}{m!}  \valabs{y-x} (x\vee y)^m\,.
  \end{align}
\end{proof}

Let us state a Lemma that we will use several times in the rest of the paper.
\begin{lemma}\label{L2}
For $\alpha>0$ and $\beta>1$ then
$T_i^{-\kappa}\in L^2(\Omega,\cA,P) $ if and only if   $i>2 \kappa$. Moreover, for every  $v>0$
there exists a  $C_v>0$ such that 
\begin{equation}\label{equivgamma}
E\Big[T_i^{-v}\Big]=\frac{C_v}{i^{v}} (1+o(1))\quad \text{ as $i\to \infty$}.
\end{equation}
\end{lemma}
\begin{proof}
We note that for $v>0$ and $i\in \N$ such that
$i-v>0$ on a 
\begin{align}\label{formexainteti}
E\Big(\frac{1}{T_i^{v}}\Big)&= \frac{\Gamma(i-v)}{\Gamma(i)}\quad  \text{si}\quad i-v>0\\
\nonumber &=\infty \quad \text{otherwise}
\end{align}
with $\Gamma$ the Gamma Euler function, i.e., 
$$\Gamma(j)=\int_0^\infty x^{j-1}\ e^{-x} dx.$$
The proof of \eqref{equivgamma} is a straightforward consequence of the generalized Stirling formula  
which gives us the asymptotics of $\Gamma$ at infinity, ie., $\Gamma(x)=\sqrt{2 \pi x} \big(x/e\big)^x\ (1+o(1))$.
\end{proof}

\section{Kolmogorov criterion of tightness}\label{kolmogcritproof}
In the  present section we prove Proposition \ref{kolmog}. To that aim, in dimension $d\geq 1$, for $\alpha>0$ and $D>0$ we let 
$\cC^\alpha([0,D]^d)$ be the set of $\alpha-$Hölderian functions on $[0,D]^d$. Clearly $\cC^\alpha([0,D]^d) \subset \cC_d$. We endow 
$\cC^\alpha([0,D]^d)$ with the infinite norm $||\cdot||_{\infty}$ and also with the norm $||\cdot ||_{\alpha}$ 
which is defined as 
\begin{equation}\label{alphanorm}
|| f||_{\alpha}:=|f(0)|+\big[f\big]_\alpha:=|f(0)|+\sup_{\x,\s\in [0,D]^d, \x\neq \s}\frac{|f(\x)-f(\s)|}{|\x-\s|_1^\alpha}.
\end{equation}
 We also denote by $\cB_\alpha(0,r)$ the ball of $\cC^\alpha([0,D]^d)$ of radius $r>0$ for the $||\cdot||_\alpha$-norm
and by  $\overline \cB_\alpha(0,r)$ its closure in $(\cC_d,||\cdot||_\infty)$.
Thanks to Arzela-Ascoli Theorem, $\overline \cB_\alpha(0,r)$ is a compact subset of $(\cC_d,||\cdot||_\infty)$.

At this stage we apply \cite[exercice 4.3.17]{stroock} to deduce from the assumption \eqref{tensionkol} that for every $\alpha \in (0,\frac{\gamma}{p})$, there exists 
$\widetilde C>0$ such that 
\begin{equation}\label{boundexpsemnorm}
\sup_{N\in \N}\ E\Big[\big[L_N\big]_{\alpha}\Big]\leq \widetilde C.
\end{equation}
At this stage, we pick an $\alpha \in (0,\frac{\gamma}{p})$. 
Since $(L_N(0))_{N\in \N}$ is tight by assumption, we claim that for every $\gep>0$ there exists 
$M_1>0$ such that for every $N\geq 1$ it comes that $P(|L_N(0)|>M_1)\leq \gep/2$. Moreover \eqref{boundexpsemnorm}
tells us that for $M_2$ large enough 
\begin{equation}\label{boundnormalpha}
\sup_{N\in \N}\, P(\big[L_N\big]_{\alpha}>M_2)\leq \frac{E\Big[\big[L_N\big]_{\alpha}\Big]}{M_2}\leq \frac{\widetilde C}{M_2}\leq \frac{\gep}{2},
\end{equation}
and therefore
\begin{align}\label{alphap}
\sup_{N\in \N}\ P\Big(|| L_N ||_{\alpha}>M_1+M_2\Big)&\leq  P\Big(|L_N(0)|>M_1\Big)+ P\Big(\big[L_N\big]_{\alpha}>M_2\Big)\\
\nonumber &\leq \frac{\gep}{2}+\frac{\gep}{2}=\gep.
\end{align}
We deduce from \eqref{alphap} that
\begin{equation}\label{prooftightlN}
\sup_{N\in \N} P\big(L_N\in \overline \cB_\alpha(0,M_1+M_2)\big)\geq 1-\gep,
\end{equation}
which guarantees us that $(L_N)_{N\in \N}$ is a tight sequence in $(\cC_d,||\cdot||_\infty)$.

\section{Proof of Proposition \ref{convgamma2N}}\label{proof-of-convgamma2N}

We set $\cF_N:=\sigma(\xi_1,\dots,\xi_N,Y_1,\dots,Y_N)$ so that $(\cF_N)_{N\geq n_1}$ is a filtration adapted to the random variable sequence 
$(\gamma_{2,N}(\x))_{N\geq n_1+1}$. Proposition \ref{convgamma2N} will be proven once we show that $(\gamma_{2,N}(\x))_{N\geq n_1+1}$ is a martingale $(\cF_N)_{N\geq n_1}$-adapted and bounded in $L^2$.  The fact that it is a martingale is easily checked, i.e., 
\begin{align}
\nonumber E\big[\gamma_{2,N+1}\,  \mid\, \cF_N\big]
&=E\big[\gamma_{2,N}\, \mid\, \cF_N\big]+E\bigg[ \frac{G_{N+1}(\x)}{T_{N+1}^\kappa}\, \mid\, \cF_N\bigg]\\
&=\gamma_{2,N}+ E[G_{N+1}(\x)] \, \,  E\bigg[ \, T_{N+1}^{-\kappa}\, \mid\, \cF_N\bigg]=\gamma_{2,N},
\end{align}
where we have used the fact that $G_{N+1}(\x)$ only depends on $Y_{N+1}$ and is therefore independent of $\sigma(\cF_N,T_{N+1})$ and is centered 
by a straightforward application of Lemma \ref{lem:loivys}.

Moreover for every $N\geq n_1+1$ we can write 
\begin{align}\label{cauchyasfix}
\nonumber   \Big\|\gamma_{2,N} \Big\|_2^2 &= \bigg\| \sum_{i=n_1+1}^{N} \frac{G_i(\x)}{T_i^\kappa} \bigg\|_2^2=E\bigg[\Big(\sum_{i=n_1+1}^{N} \frac{G_i(\x)}{T_i^\kappa} \Big)^2\bigg]\\
&=\sum_{i=n_1+1}^{N} E\Bigg[ \frac{G_i(\x)^2}{T_i^{2\kappa}}\Bigg]+2 \sum_{n_1+1\leq i<j\leq N} E\Bigg[\frac{G_i(\x) G_j(\x)}{(T_i T_j)^\kappa}\Bigg].
\end{align}
We note that $G_i(\x), G_j(\x)$ and $\frac{1}{(T_i T_j)^{\kappa}}$ are independent and integrable. Moreover, $G_i(\x)$ and $G_j(\x)$ are bounded and centered and therefore the 
second term in the r.h.s. in 
\eqref{cauchyasfix} is zero. Thus, with the help of \eqref{equivgamma} and with the inequality $\kappa>\undemi$ we obtain
\begin{align}\label{conv1}
\nonumber \Big\| \sum_{i=n_1+1}^{N} \frac{G_i(\x)}{T_i^\kappa} \Big \|_2^2&\leq E\Big( G_1(\x)^2\Big)
\sum_{i=n_1+1}^{\infty} E\Bigg[ \frac{1}{T_i^{2\kappa}}\Bigg]\\
&\leq \text{Cste} \sum_{i=n_1+1}^{\infty} \frac{1}{i^{2\kappa}} <\infty.
\end{align}
Thus, \eqref{conv1} guarantees us that $(\gamma_{2,N}(\x))_{N>n_1+1}$ is a martingale  bounded in $L^2(\Omega,\cA,P)$, which proves the almost sure convergence towards a random variable $\gamma_{2,\infty}\in L^2(\Omega,\cA,P)$. The convergence also holds true in $L^2$.
$\Box$

\section{A specific example of variables in the domain of attraction
  of Normal law}
\begin{lemma}\label{lem:att_par_normal}
  Let $\xi$ be a random variable such that
  \begin{equation}
    \prob{\pm\xi >  x} \sim_{x\to +\infty} c_\pm x^{-2}.
  \end{equation}
  Then if $\mu=\esp{\xi}$, $\xi,\xi_1,...,\xi_n$ are i.i.d., the partial
  sum $S_n = \xi_1+ \cdots +\xi_n$ is in the domain of attraction of a
  normal law. More precisely
  \begin{equation}
    \unsur{\sqrt{n\log n}} (\xi_1 + \cdots + \xi_n - n \mu) \cvloin
      \Nrond(0,c_+ + c_-)\,.
  \end{equation}
\end{lemma}
\begin{proof}
  This Lemma is just a straightforward consequence of \citet[Theorem 8.3.1]{BGT87}
  as we can derive the asymptotics of the truncated variance
  
  \begin{align*}
    V(x) &= \esp{\xi^2 \un{\valabs{\xi} \le x}} \\
    &= \intof \prob{\xi^2> t,\valabs{\xi} \le x}\, dt\\
    &= \int_0^{x^2} \prob{ \sqrt{t}<
      \xi\leq x }\, dt + \int_0^{x^2} \prob{\sqrt{t}<-\xi\leq x}\, dt\\
    &= 2 \int_0^x \prob{u< \xi \le x} u \, du +  2 \int_0^x \prob{u< -\xi \le x} u \, du \,.
  \end{align*}
Given $\ee>0$, there exists $t_0$ such that for $t\ge t_0$, we have
\begin{equation}
  c_+(1-\ee) t^{-2} \le \prob{\xi >t} \le  c_+(1+\ee) t^{-2}
\end{equation}
Therefore, if $x > t_0$, then
\begin{align*}
  2 \int_{t_0}^x \prob{u< \xi \le x} u \, du &=  2 \int_{t_0}^x
  (\prob{\xi>u}-\prob{\xi > x}) u \, du\\
  &\le 2 c_+  (1+\ee)\int_{t_0}^x \frac{du}{u}
  - 2 c_+ x^{-2}  (1-\ee)\int_{t_0}^x u \,du \\
  &=  2 c_+  (1+\ee) \log(\frac{x}{t_0})  -  c_+   (1-\ee) \frac{x^2-t_0^2}{x^2}\,.
\end{align*}
Hence,
\begin{equation}
  \limsup_{x \to +\infty} \unsur{\log x} 2 \int_0^x \prob{u< \xi
    \le x} u \, du \le 2 c_+ (1+\ee)\,,
\end{equation}
and letting $\ee\to 0$ yields
\begin{equation}
  \limsup_{x \to +\infty} \unsur{\log x} 2 \int_0^x \prob{u< \xi
    \le x} u \, du \le 2 c_+\,.
\end{equation}
It is easy to prove similarly that
\begin{equation}
   \liminf_{x \to +\infty} \unsur{\log x} 2 \int_0^x \prob{u< \xi
    \le x} u \, du \ge 2 c_+\,.
\end{equation}
Eventually, we establish that
\begin{equation}
  \lim_{x\to +\infty} \unsur{\log x} V(x) = 2(c_+ + c_-)
\end{equation}
and therefore $V(x)$ is slowly varying.
  
  It is now easy to see that the norming sequence $a_n = \sqrt{n \log
    n}$ satisfies \citet[(8.3.7)]{BGT87}:
  \begin{align*}
\lim_{n\to +\infty} \frac{n}{a_n^2} V(a_n) = c_+ + c_-\,.
\end{align*}
Since $\esp{\xi}=\mu$ has first moment we can choose $b_n = n\mu$ to
satisfy
\begin{align*}
\lim_{n\to \infty} a_n (S_n -b_n) \underset{\text{law}}{=}\, \Nrond(0,\sigma^2(c_+ + c_-))\,.
\end{align*}
Transforming $\xi $ into $a \xi$ for $a>0$, we see that the function
$\sigma^2(c)$ is linear $\sigma^2(c) = c \sigma^2_1$. In the case $\xi
= \epsilon U^{-1/2}$ with $\prob{\xi = \pm 1}=\undemi$ and $U$ uniform
on $(0,1)$, we have $c_\pm=\undemi$ and  can compute exactly
$\sigma^2_1=1$, by using a Taylor expansion of the caracteristic
function of $\xi$ near $0$.
\end{proof}
%
%
%
\bibliography{references}
\bibliographystyle{plainnat}

 \end{document}